\documentclass{article}

\title{Nested Volume-Surface Integral Equations for Acoustics}

\author{Danilo Aballay\thanks{Institute for Mathematical and Computational Engineering, School of Engineering and Faculty of Mathematics, Pontificia Universidad Católica de Chile, Santiago, Chile.} \and Elwin van 't Wout\footnotemark[1] \thanks{Contact: e.wout@uc.cl}}

\date{July 7, 2026}

\usepackage[utf8]{inputenc}
\usepackage{amsmath, amssymb, amsfonts, amsthm}
\usepackage{latexsym}
\usepackage{tikz}
\usepackage[numbers]{natbib}
\usepackage{doi}
\usepackage{float}
\usepackage{booktabs}
\usepackage{fancyhdr}

\newcommand{\norm}[1]{\left\lVert #1 \right\rVert}
\renewcommand{\d}{\,\mathrm{d}}
\newcommand{\dx}{\d \mathbf{x}}
\newcommand{\dy}{\d \mathbf{y}}
\newcommand{\Rthree}{\mathbb{R}^3}
\newcommand{\normal}{\hat{\mathbf{n}}}
\newcommand{\trace}{\gamma}
\newcommand{\traceD}{\trace_\mathrm{D}}
\newcommand{\traceN}{\trace_\mathrm{N}}
\newcommand{\traceDe}{\traceD^+}
\newcommand{\traceDi}{\traceD^-}
\newcommand{\traceDei}{\traceD^\pm}
\newcommand{\traceNe}{\traceN^+}
\newcommand{\traceNi}{\traceN^-}
\newcommand{\traceNei}{\traceN^\pm}
\newcommand{\ID}{M}
\newcommand{\SL}{V}
\newcommand{\DL}{K}
\newcommand{\AD}{T}

\theoremstyle{plain}
\newtheorem{theorem}{Theorem}
\newtheorem{lemma}[theorem]{Lemma}

\theoremstyle{definition}

\theoremstyle{remark}
\newtheorem{remark}[theorem]{Remark}

\usetikzlibrary{decorations.markings, arrows, arrows.meta}

\tikzset{
    midar/.style 2 args={
        very thick,
        decoration={name=markings,
        mark=at position .55 with {\arrow{latex}},
        mark=at position 0 with {\fill circle (2pt);},
        mark=at position 1 with {\fill circle (2pt);}}
        ,postaction=decorate,
    },
}

\begin{document}

\maketitle

\begin{abstract}
The simulation of high-frequency acoustic wave propagation in unbounded domains with local heterogeneous materials and high-contrast interfaces poses significant challenges to numerical methods. The volume-surface integral equation (VSIE) method is an attractive approach as it automatically satisfies the radiation condition at infinity via Green's functions, handles heterogeneous materials via Newton potentials, and models scattering at high-contrast interfaces via surface integral operators. However, its effectiveness in practical simulations has been limited by high computational costs, sensitivity to sharp interfaces, and insufficient computational verification. This study extends the applicability of VSIE by deriving integral formulations for nested heterogeneous materials with parameter jumps at interfaces. We also develop extensive benchmarks against coupled finite-element and boundary-element methods to verify the VSIE's accuracy and mesh convergence. The various benchmarks using open-source software demonstrate the effectiveness of VSIE for large-scale acoustic simulations.
\end{abstract}

\section{Introduction}

The Helmholtz equation is an ubiquitous model for harmonic wave propagation and the mathematical backbone of many computational acoustics simulations~\cite{morse1986theoretical, hamilton2024nonlinear, garrett2020understanding}. This partial differential equation reads
\begin{equation}
    -\rho \,\nabla \cdot \left(\frac{1}{\rho} \nabla p\right) - \left(\frac{\omega}{c}\right)^2 p = f
    \label{eq_helmholtz_general}
\end{equation}
for the unknown pressure $p(\mathbf{x})$, known source $f(\mathbf{x})$ with angular frequency $\omega$, and a material with mass density $\rho(\mathbf{x})$ and speed of sound $c(\mathbf{x})$, all at the position $\mathbf{x} \in \Rthree$. This study considers a domain consisting of two nested subdomains embedded in an unbounded exterior region. The material parameters $\rho$ and $c$ may be heterogeneous in the interior domains and discontinuous across interfaces. Except for special choices of material parameters, the Helmholtz equation for multiple, heterogeneous materials does not have an analytical solution and efficient numerical methods must be used~\cite{lahaye2017modern}.

Solving the acoustic transmission model with volumetric methods such as the finite element method (FEM), finite volume method, and finite difference method is challenging. Handling scattering in the unbounded domain necessitates artificial boundary conditions~\cite{engquist1977absorbing} or absorbing layers~\cite{berenger1994perfectly} to limit the computational domains and avoid spurious reflections. Also, for a fixed number of grid elements per wavelength, the size of the sparse discretization matrix scales with the third power of the frequency~\cite{logg2012automated}. Worse, these algorithms suffer from the pollution effect and need increasingly finer mesh resolutions or high-order discretizations at high frequencies to maintain accuracy~\cite{babuska1997pollution, appelo2026rule}.

An alternative approach to solving the Helmholtz equation is to reformulate it as an integral equation via representation formulas based on Green's functions. When the material parameters are piecewise constant, the Green's identities can be employed to rewrite the volumetric Helmholtz equation into boundary integral equations at the material interfaces~\cite{nedelec2001acoustic}. The advantage of this boundary integral approach is that the number of degrees of freedom in the surface meshes scales only quadratically, and the radiation condition at infinity is automatically satisfied~\cite{sauter2011boundary}. However, the discretization matrix of the boundary element method (BEM) is dense and must be solved using acceleration techniques such as the fast multipole method~\cite{greengard1987fast} or hierarchical matrix compression~\cite{borm2010efficient}. Importantly, few elements per wavelength suffice to achieve accurate solutions, even at high frequencies~\cite{marburg2002six, galkowski2023does}. However, the BEM is fundamentally limited to piecewise-constant material parameters for which free-space Green's functions are available in each homogeneous subdomain.

In our targeted scenario of heterogeneous materials embedded in unbounded domains, a coupled FEM-BEM approach may be adopted, where the FEM handles the heterogeneous interior and the BEM the radiation condition at infinity~\cite{johnson1980coupling}. While feasible in many situations, scaling coupled FEM-BEM algorithms to large-scale simulations at high frequencies is challenging due to the ill-conditioning of the matrix with mixed sparse and dense blocks and the stringent demands on mesh resolution~\cite{wout2022fembem, wout2024nonconforming}.

This article proposes an integral equation approach for efficiently modeling acoustic wave propagation in heterogeneous materials in unbounded domains. The main idea is to write the Helmholtz equation as a perturbed integral equation in locally heterogeneous domains. While the exterior scattering problem can be handled with a representation formula for the homogeneous Green's function, volume integral operators (also known as Newton potentials) remain in the heterogeneous bounded domains~\cite{steinbach2008numerical}. A well-studied case in the simpler setting of a constant material density is the Lippmann-Schwinger equation~\cite{kirsch2009operator}. However, high contrasts in density across material interfaces are a common feature of many acoustic models. In those cases, the volume integral equations (VIE) in the heterogeneous domains must be coupled with surface integral equations (SIE) at high-contrast interfaces, yielding coupled volume-surface integral equations (VSIE).

The VSIE approach has been applied to various wave propagation models. For example, a variant of the VIE is often used in optics~\cite{hage1991scattering}, where Maxwell's equation with constant permeability and a smooth permittivity can be written as a well-defined integral equation~\cite{beurden2008well, costabel2012essential}, and solved with the discrete dipole algorithm~\cite{draine1994discrete, yurkin2007discrete}. In electromagnetics, the VSIE is especially efficient in modeling scattering at perfect conductors with thin dielectric coatings~\cite{schaubert1984tetrahedral, lu2000coupled}. Furthermore, various acceleration algorithms have been designed for the dense matrix arithmetic~\cite{chew2001fast}, including the fast multipole algorithm~\cite{wang2021fast}, hierarchical matrix compression~\cite{wang2022fast}, and the adaptive integral method~\cite{wei2014more1} that employs fast Fourier transformations~\cite{polimeridis2014stable}. Among the few examples of VIE applications to acoustics, we highlight the accelerated implementation of a biomedical model for a heterogeneous head model~\cite{bleszynski2008fast}. There is also a solid mathematical foundation for the VSIE. For example, the VSIE has a unique solution~\cite{martin2003acoustic}, the Galerkin weak formulation fits within Sobolev spaces~\cite{costabel2015spectrum}, and the integral operators are coercive~\cite{labarca2024volume, labarca2025coupled}.

The adoption of the VSIE in computational acoustics has been much more limited than the alternative FEM, BEM, and coupled FEM-BEM approaches. This study tackles two reasons for the low popularity of the VSIE in computational acoustics. First, we extend the VSIE formulation to a larger set of geometric scenarios, namely nested domains with multiple high-contrast interfaces. Second, we benchmark the VSIE's accuracy against common FEM-BEM algorithms with open-source software. This computational verification is an essential step to guarantee a consistent solution, evaluate comparative characteristics, and design novel formulations. Finally, another reason for the limited use of the VSIE is the computational footprint required to handle dense matrices that scale cubically with the number of grid elements. We leave computational efficiency outside the scope of this manuscript, but mention that all considered formulations support fast arithmetic with hierarchical matrix compression~\cite{almuna2026full}.

We emphasize that the verification and validation of mathematical models and computational algorithms are essential for the credibility of scientific computing~\cite{oberkampf2010verification}. Specifically, numerical methods must correctly solve the underlying partial differential equation with sufficient accuracy~\cite{babuska2004verification}. Here, we will verify the proposed VSIE method by performing computational benchmarks against analytical solutions and highly accurate numerical methods~\cite{oberkampf2002verification}. We choose various geometries and material parameters with increasing complexity for the density. The VSIE's results are then compared with independent implementations of alternative approaches, such as spherical harmonics, BEM, and FEM-BEM, where possible. 

We provide a full derivation of the VSIE for nested domains in Section~\ref{sec:formulation} and present the verification results in Section~\ref{sec:results}.

\section{Formulation}
\label{sec:formulation}

Section~\ref{sec_helmholtz} introduces the equations of motion for harmonic acoustic waves. The nested VSIE model for the Helmholtz equation in unbounded domains with local heterogeneities will be derived in Section~\ref{sec_vsie}. Then, Section~\ref{sec_galerkin} provides the weak formulation and numerical discretization of the VSIE.

\subsection{Helmholtz equation}
\label{sec_helmholtz}

This manuscript studies harmonic acoustic wave propagation in unbounded regions with local heterogeneities in the material parameters for nested domains. To keep the presentation of the formulation concise but sufficiently general, we consider two bounded subdomains, one inside the other, as sketched in Figure~\ref{fig_geometry}. Precisely, $\Omega_0$ is an unbounded domain exterior to the bounded domain~$\Omega_1$ inside which $\Omega_2$ is located. All subdomains are three-dimensional open volumes. Hence, the union $\Omega = \Omega_0 \cup \Omega_1 \cup \Omega_2$ covers the entire space $\mathbb{R}^3$ except for the boundaries. We denote the interface between domains $\Omega_0$ and $\Omega_1$ as $\Gamma_1$, and $\Gamma_2$ is the $\Omega_1$-$\Omega_2$ interface. The surfaces are assumed to be Lipschitz continuous with a unit normal vector~$\normal_j$ pointing outwards of subdomain~$\Omega_j$ for $j=0,1,2$. Notice that this yields two normals at each interface (see Figure~\ref{fig_geometry}).

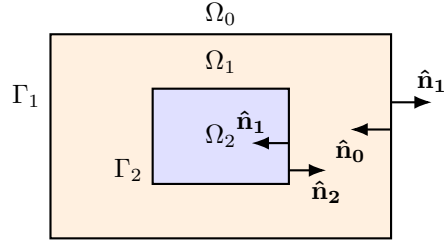
\begin{figure}[htbp]
    \centering
    \begin{tikzpicture}[scale=0.9, thick]
        \fill[orange!12] (0,0) rectangle (5,3);
        \fill[blue!12]   (1.5,0.8) rectangle (3.5,2.2);
        \draw (0,0) rectangle (5,3);          
        \draw (1.5,0.8) rectangle (3.5,2.2);  
        \node at (2.5,3.3) {$\Omega_0$};
        \node at (2.5,2.6) {$\Omega_1$};
        \node at (2.5,1.5) {$\Omega_2$};
        \node[left] at (0,2.1)   {$\Gamma_1$};
        \node[left] at (1.5,1.0) {$\Gamma_2$};
        \draw[-{Latex}] (3.5,1.0) -- (4.05,1.0) node[below] {$\mathbf{\hat{n}_2}$}; 
        \draw[-{Latex}] (3.5,1.4) -- (2.95,1.4) node[above] {$\mathbf{\hat{n}_1}$}; 
        \draw[-{Latex}] (5,1.6) -- (4.40,1.6) node[below] {$\mathbf{\hat{n}_0}$};   
        \draw[-{Latex}] (5,2.0) -- (5.60,2.0) node[above] {$\mathbf{\hat{n}_1}$};   
    \end{tikzpicture}
    \caption{A sketch of the geometry. The domain $\Omega_0$ is unbounded, and $\Omega_2$ is inside $\Omega_1$. The normals point outward from the respective subdomains.}
    \label{fig_geometry}
\end{figure}

As is common practice in computational engineering, each subdomain represents a different material. Hence,
\begin{align}
    \label{eq_density}
    \rho(\mathbf{x}) = \begin{cases} \rho_0 & \text{in } \Omega_0, \\ \rho_1(\mathbf{x}) & \text{in } \Omega_1, \\ \rho_2(\mathbf{x}) & \text{in } \Omega_2; \end{cases} \\
    \label{eq_speed_of_sound}
    c(\mathbf{x}) = \begin{cases} c_0 & \text{in } \Omega_0, \\ c_1(\mathbf{x}) & \text{in } \Omega_1, \\ c_2(\mathbf{x}) & \text{in } \Omega_2; \end{cases}
\end{align}
denote the mass density and speed of sound of the materials, respectively. Here, we assume the unbounded exterior domain to be `homogeneous', interpreted as having constant material parameters, and the bounded domains are `heterogeneous', meaning location-dependent functions for the density and speed of sound. Furthermore, we assume $c(\mathbf{x}) \in \mathcal{C}(\Omega)$ and $\rho(\mathbf{x}) \in \mathcal{C}^1(\Omega)$, continuous and continuously differentiable inside each subdomain, respectively. Notice that this allows for discontinuities at the interfaces $\Gamma_1$ and $\Gamma_2$, as $\Omega$ excludes these surfaces. Finally, the material parameters are bounded and positive, that is, $0 < \rho_\mathrm{min} \leq \rho(\mathbf{x}) \leq \rho_\mathrm{max}$ and $0 < c_\mathrm{min} \leq c(\mathbf{x}) \leq c_\mathrm{max}$ for all $\mathbf{x} \in \Omega$.

An $e^{-\imath \omega t}$ time component will be assumed for the harmonic wave propagation. Furthermore, the acoustic pressure and normal particle velocity are continuous across material interfaces. At infinity, the Sommerfeld radiation condition forces outgoing fields. Finally, the general Helmholtz equation~\eqref{eq_helmholtz_general} can be written for our setting as the boundary value problem
\begin{align}
    &-\nabla^2 \, p(\mathbf{x}) - k_0^2 p(\mathbf{x}) = f(\mathbf{x}), && \mathbf{x} \in \Omega_0;
    \label{eq_helmholtz_exterior} \\
    &-\rho_j(\mathbf{x}) \,\nabla \cdot \left(\frac{1}{\rho_j(\mathbf{x})} \nabla p(\mathbf{x})\right) - \left(\frac{\omega}{c_j(\mathbf{x})}\right)^2 p(\mathbf{x}) = f(\mathbf{x}), && \mathbf{x} \in \Omega_j \text{ for } j=1,2;
    \label{eq_helmholtz_subdomains} \\
    &\traceDe p(\mathbf{x}) = \traceDi p(\mathbf{x}), && \mathbf{x} \in \Gamma_j \text{ for } j=1,2;
    \label{eq_transmission_pressure} \\
    &\frac1{\traceDe\rho(\mathbf{x})} \, \traceNe p(\mathbf{x}) = \frac1{\traceDi\rho(\mathbf{x})} \, \traceNi p(\mathbf{x}), && \mathbf{x} \in \Gamma_j \text{ for } j=1,2;
    \label{eq_transmission_velocity} \\
    &\lim_{\lvert\mathbf{x}\rvert \to \infty} \lvert\mathbf{x}\rvert \left(\nabla p(\mathbf{x}) \cdot \lvert\mathbf{x}\rvert - \imath k_0 p(\mathbf{x})\right) = 0.
    \label{eq_sommerfeld}
\end{align}
Here, $\imath$ denotes the imaginary unit, $k_0 = \omega/c_0$ the exterior wavenumber, and the Dirichlet and Neumann traces are given by
\begin{align}
    \traceDei v(\mathbf{x}) &= \lim_{\mathbf{y} \to \mathbf{x}} v(\mathbf{y}) & \text{for } \mathbf{x} \in \Gamma \text{ and } \mathbf{y} \in \Omega^\pm, \\
    \traceNei v(\mathbf{x}) &= \lim_{\mathbf{y} \to \mathbf{x}} \nabla v(\mathbf{y}) \cdot \normal(\mathbf{x}) & \text{for } \mathbf{x} \in \Gamma \text{ and } \mathbf{y} \in \Omega^\pm,
\end{align}
where $\Omega^+$ and $\Omega^-$ are the domains exterior and interior to the interface $\Gamma$, respectively. The known $f(x)$ models the acoustic source. Since most physics and engineering applications consider sources of finite energy and bounded size, $f(x)$ will be assumed to be integrable and compactly supported.

\subsection{Volume-surface integral formulation}
\label{sec_vsie}

This section derives the continuous integral formulation for the nested Helmholtz problem, which will be solved subsequently with numerical methods. We remark that, although the formulation is novel, the operators are well-studied for similar formulations. See, e.g., \cite{colton1998inverse, steinbach2008numerical, sauter2011boundary, costabel2012essential, labarca2024volume}, for more information about their functional analysis.

\subsubsection{Exterior Helmholtz equation}

Remember that the unbounded exterior domain is homogeneous and that heterogeneities are localized to the bounded subdomains. These geometrical characteristics are essential to the VSIE design, as it solves a perturbation of the exterior Helmholtz equation. For this purpose, let us define the derived material parameters $\alpha$ and $\beta$ as
\begin{align}
    \alpha(\mathbf{x}) &= \frac{\rho_0}{\rho(\mathbf{x})} - 1,
    \label{eq_alpha} \\
    \beta(\mathbf{x}) &= \frac{\rho_0}{\rho(\mathbf{x})} \left(\frac{\omega}{c(\mathbf{x})}\right)^2 - \left(\frac{\omega}{c_0}\right)^2,
    \label{eq_beta}
\end{align}
for $\mathbf{x} \in \Omega$. A direct consequence of these definitions is that $\alpha(\mathbf{x}) = 0$ and $\beta(\mathbf{x}) = 0$ in the unbounded exterior domain~$\Omega_0$. Hence, these functions have compact support in $\Omega_1 \cup \Omega_2$, the interior domains. Also, they inherit the regularity and boundedness properties from the physical parameters.

Then, multiplying the Helmholtz equations~\eqref{eq_helmholtz_exterior} and~\eqref{eq_helmholtz_subdomains} by $\rho_0/\rho(\mathbf{x})$ yields
\begin{align*}
    -\nabla \cdot \left(\frac{\rho_0}{\rho(\mathbf{x})} \nabla p(\mathbf{x})\right) - \frac{\rho_0}{\rho(\mathbf{x})} \left(\frac{\omega}{c(\mathbf{x})}\right)^2 p(\mathbf{x}) = \frac{\rho_0}{\rho(\mathbf{x})} f(\mathbf{x}), \quad \mathbf{x} \in \Omega,
\end{align*}
and substituting $\alpha$ and $\beta$ gives
\begin{align*}
    -\nabla \cdot \left(\left(\alpha(\mathbf{x}) + 1\right) \nabla p(\mathbf{x})\right) - \left(\beta(\mathbf{x}) + \left(\frac{\omega}{c_0}\right)^2\right) p(\mathbf{x}) = \frac{\rho_0}{\rho(\mathbf{x})} f(\mathbf{x}), \quad \mathbf{x} \in \Omega.
\end{align*}
Moving all heterogeneous parameters to the right yields
\begin{align}
    -\nabla^2 p(\mathbf{x}) - k_0^2 p(\mathbf{x}) = \frac{\rho_0}{\rho(\mathbf{x})} f(\mathbf{x}) + \nabla \cdot \left(\alpha(\mathbf{x}) \nabla p(\mathbf{x})\right) + \beta(\mathbf{x}) p(\mathbf{x}), \quad \mathbf{x} \in \Omega,
    \label{eq_perturbed_helmholtz}
\end{align}
the \emph{perturbed} Helmholtz equation.

The advantage of writing the Helmholtz equation in this perturbed form is that the left-hand side is a homogeneous Helmholtz equation, and the pressure terms on the right-hand side are localized to the interior domain. Hence, Eq.~\eqref{eq_perturbed_helmholtz} can be interpreted as a local perturbation to a free-field problem.

\subsubsection{Green's function}

Since the left-hand side of Eq.~\eqref{eq_perturbed_helmholtz} is a homogeneous Helmholtz equation, let us design integral equations with the exterior Green's functions. Precisely, for a constant $\kappa$ and $\delta$ denoting the Dirac delta distribution, a free-space homogeneous Helmholtz equation
\begin{equation}
    -\left(\nabla_\mathbf{x}^2 + \kappa^2\right) G_\kappa(\mathbf{x}, \mathbf{y}) = \delta(\mathbf{x} - \mathbf{y}), \quad \mathbf{x}, \mathbf{y} \in \Rthree,
    \label{eq_helmholtz_dirac}
\end{equation}
has a fundamental solution~\cite{duffy2015green}, given by the well-known Green's function
\begin{equation}
    G_\kappa(\mathbf{x}, \mathbf{y}) = \frac{e^{\imath \kappa \lvert\mathbf{x} - \mathbf{y}\rvert}}{4\pi \lvert\mathbf{x} - \mathbf{y}\rvert}, \quad \mathbf{x} \ne \mathbf{y}.
    \label{eq_greens_function}
\end{equation}
Here, we use the notation $\nabla_\mathbf{x}$ to emphasize that the derivatives are taken with respect to the variable~$\mathbf{x}$. A key theorem in the design of integral equations is that solutions of the Helmholtz equation can be expressed as a convolution between the Green's function and the source term (e.g.,~\cite[Lemma~1]{costabel2015spectrum}). That is, if $v(\mathbf{x})$ is an integrable function with compact support, then the solution $u(\mathbf{x})$ of the homogeneous Helmholtz equation,
\begin{equation}
    -\left(\nabla^2 + \kappa^2\right) u(\mathbf{x}) = v(\mathbf{x}),
\end{equation}
is given by the distribution
\begin{equation}
    u(\mathbf{x}) = \left(G_\kappa \star v\right)(\mathbf{x}) = \int_{\Rthree} \frac{e^{\imath \kappa \lvert\mathbf{x} - \mathbf{y}\rvert}}{4\pi \lvert\mathbf{x} - \mathbf{y}\rvert} v(\mathbf{y}) \dy, \quad \mathbf{x} \in \Rthree
    \label{eq_theorem_convolution}
\end{equation}
where $\star$ denotes the convolution operator. This convolution is also known as the Newton potential (e.g.,~\cite[Section~3.1.1]{sauter2011boundary}).

\subsubsection{Field representation}

The fundamental step in deriving the acoustic VSIE is to represent the pressure field as a convolution with the Green's function. This step can be achieved by applying the above lemma to the Helmholtz equation, yielding a VIE for the unknown pressure field~$p$, given a known incident pressure field $p_\mathrm{inc}$.

\begin{theorem}
\label{theorem_convolution}
The solution $p(\mathbf{x})$ of the perturbed Helmholtz equation~\eqref{eq_perturbed_helmholtz} is implicitly given by
\begin{align}
    p(\mathbf{x}) &- \int_{\Omega_1 \cup \Omega_2} G_{k_0}(\mathbf{x},\mathbf{y}) \beta(\mathbf{y}) p(\mathbf{y}) \dy \nonumber \\
    &- \int_{\Omega_1 \cup \Omega_2} G_{k_0}(\mathbf{x},\mathbf{y}) \, \nabla_{\mathbf{y}} \cdot \left(\alpha(\mathbf{y}) \nabla_{\mathbf{y}} p(\mathbf{y})\right) \dy \nonumber \\
    &= p_\mathrm{inc}(\mathbf{x}), 
    &\mathbf{x} \in \Omega,
    \label{eq_vie_general}
\end{align}
the general VIE, where the known incident wave field is given by
\begin{equation}
    p_\mathrm{inc}(\mathbf{x}) = \int_{\Rthree} G_{k_0}(\mathbf{x},\mathbf{y}) \frac{\rho_0}{\rho(\mathbf{y})} f(\mathbf{y}) \dy, \quad \mathbf{x} \in \Rthree,
    \label{eq_pinc}
\end{equation}
for the source $f(\mathbf{x})$.
\end{theorem}
\begin{proof}
Let us define
\begin{equation}
    v(\mathbf{x}) = \frac{\rho_0}{\rho(\mathbf{x})} f(\mathbf{x}) + \beta(\mathbf{x}) p(\mathbf{x}) + \nabla \cdot \left(\alpha(\mathbf{x}) \nabla p(\mathbf{x})\right), \quad \mathbf{x} \in \Omega,
\end{equation}
so that the perturbed Helmholtz equation~\eqref{eq_perturbed_helmholtz} simplifies to
\begin{align*}
    -\nabla^2 p(\mathbf{x}) - k_0^2 p(\mathbf{x}) = v(\mathbf{x}), \quad \mathbf{x} \in \Omega.
\end{align*}
To apply the convolution theorem, $v(\mathbf{x})$ must be integrable with compact support. First, $f(\mathbf{x})$ is integrable and has a compact support since it represents the acoustic source. Second, the unknown $p(\mathbf{x})$, and its second-order derivatives, must be integrable because it represents the acoustic pressure field. Third, the derived parameters $\alpha(\mathbf{x})$, $\nabla\alpha(\mathbf{x})$, and $\beta(\mathbf{x})$ are bounded, continuous, and of compact support in the interior domain, by construction. By linearity, $v(\mathbf{x})$ is integrable with compact support.

Now, we can apply the convolution theorem~\eqref{eq_theorem_convolution} to the perturbed Helmholtz equation~\eqref{eq_perturbed_helmholtz} to write the pressure field as
\begin{align*}
    p = G_{k_0} \star \left(\frac{\rho_0}{\rho} f + \beta p + \nabla \cdot \left(\alpha \nabla p\right)\right) \text{ in } \Omega.
\end{align*}
By linearity of the convolution operator,
\begin{align*}
    p - G_{k_0} \star \left(\beta p\right) - G_{k_0} \star \left(\nabla \cdot \left(\alpha \nabla p\right)\right) = G_{k_0} \star \left(\frac{\rho_0}{\rho} f\right) \text{ in } \Omega.
\end{align*}
Writing out the convolution and substituting the definition~\eqref{eq_pinc} yields Eq.~\eqref{eq_vie_general}.
\end{proof}

\subsubsection{Reducing the differentiation order}

While the VIE~\eqref{eq_vie_general} could be discretized directly, this yields numerical challenges in handling the divergence term $\nabla_{\mathbf{y}} \cdot \left(\alpha(\mathbf{y}) \nabla_{\mathbf{y}} p(\mathbf{y})\right)$ in the second convolution. First, while $\alpha(\mathbf{x}) \in \mathcal{C}^1(\Omega)$ allows the divergence to be computed, this may yield sharp gradients in heterogeneous materials on a coarse mesh. Second, the second-order derivatives of the unknown pressure $p(\mathbf{x})$ must be well represented on the mesh, which necessitates the use of high-order numerical schemes.

While the natural functional setting for the pressure of the underlying continuous PDE \eqref{eq_helmholtz_exterior} and~\eqref{eq_helmholtz_subdomains} is the Sobolev space $H^1(\Omega)$, the integral operator itself extends to the $L^2(\Omega)$ function space~\cite{costabel2015spectrum}. Nevertheless, discretizing this formulation with low-order basis functions leads to ill-conditioning and poor numerical accuracy at high density contrasts~\cite{bleszynski2008contrast, kirsch2009operator}. Hence, we will perform a series of algebraic manipulations to reduce the differentiation order and achieve a VSIE formulation suitable for numerical evaluation in $L^2(\Omega)$.

Let us start manipulating the VIE by swapping the order of the convolution and divergence.

\begin{lemma}
The VIE~\eqref{eq_vie_general} is equivalent to
\begin{align}
    p(\mathbf{x}) &- \int_{\Omega_1 \cup \Omega_2} G_{k_0}(\mathbf{x},\mathbf{y}) \beta(\mathbf{y}) p(\mathbf{y}) \dy \nonumber \\
    &- \nabla_{\mathbf{x}} \cdot \int_{\Omega_1 \cup \Omega_2} G_{k_0}(\mathbf{x},\mathbf{y}) \alpha(\mathbf{y}) \nabla_{\mathbf{y}} p(\mathbf{y}) \dy \nonumber \\
    &= p_\mathrm{inc}(\mathbf{x}),
    &&\mathbf{x} \in \Omega.
    \label{eq_vie_convolution}
\end{align}
\end{lemma}
\begin{proof}
This is a direct consequence of the property $\partial(u \star v) = u \star \partial v$ for the convolution operator $\star$, where $\partial$ denotes differentiation~\cite{friedlander1998introduction}.
\end{proof}

Now, let us apply the well-known integration-by-parts formula to the gradient term. This will yield a VSIE with surface integrals at the material interfaces.

\begin{lemma}
The VIE~\eqref{eq_vie_convolution} is equivalent to the VSIE
\begin{align}
    p(\mathbf{x}) &- \int_{\Omega_1 \cup \Omega_2} G_{k_0}(\mathbf{x},\mathbf{y}) \beta(\mathbf{y}) p(\mathbf{y}) \dy \nonumber \\
    &+ \nabla_{\mathbf{x}} \cdot \int_{\Omega_1 \cup \Omega_2} \nabla_{\mathbf{y}} \left(G_{k_0}(\mathbf{x},\mathbf{y}) \alpha(\mathbf{y})\right) p(\mathbf{y}) \dy \nonumber \\
    &- \nabla_{\mathbf{x}} \cdot \int_{\Gamma_1} G_{k_0}(\mathbf{x},\mathbf{y}) \alpha_1(\mathbf{y}) p(\mathbf{y}) \normal_1(\mathbf{y}) \dy \nonumber \\
    &- \nabla_{\mathbf{x}} \cdot \int_{\Gamma_2} G_{k_0}(\mathbf{x},\mathbf{y}) \left(\alpha_2(\mathbf{y}) - \alpha_1(\mathbf{y})\right) p(\mathbf{y}) \normal_2(\mathbf{y}) \dy \nonumber \\
    &= p_\mathrm{inc}(\mathbf{x}),
    &&\mathbf{x} \in \Omega.
    \label{eq_vsie_divergence}
\end{align}
\end{lemma}
\begin{proof}
Integration by parts provides
\begin{align}
    \int_{\Omega_1 \cup \Omega_2} G_{k_0}(\mathbf{x},\mathbf{y}) \alpha(\mathbf{y}) \nabla_{\mathbf{y}} p(\mathbf{y}) \dy
    =& -\int_{\Omega_1 \cup \Omega_2} \nabla_{\mathbf{y}} \left(G_{k_0}(\mathbf{x},\mathbf{y}) \alpha(\mathbf{y})\right) p(\mathbf{y}) \dy \nonumber \\
    &+ \int_{\Gamma_1} G_{k_0}(\mathbf{x},\mathbf{y}) \, \alpha_1(\mathbf{y}) p(\mathbf{y}) \normal_1(\mathbf{y}) \dy \nonumber \\
    &+ \int_{\Gamma_2} G_{k_0}(\mathbf{x},\mathbf{y}) \, \alpha_1(\mathbf{y}) p(\mathbf{y}) \normal_1(\mathbf{y}) \dy \nonumber \\
    &+ \int_{\Gamma_2} G_{k_0}(\mathbf{x},\mathbf{y}) \, \alpha_2(\mathbf{y}) p(\mathbf{y}) \normal_2(\mathbf{y}) \dy.
\end{align}
Here, the subscripts $\alpha_j(\mathbf{x}) = \rho_0/\rho_j(\mathbf{x}) - 1$ for $j=1,2$ emphasize that the material parameters at the interface are limit values from inside the subdomains. Mathematically, $\alpha_1(\mathbf{y}) = \lim_{\Omega_1 \ni \mathbf{z} \to \mathbf{y}} \alpha(\mathbf{z})$ for $\mathbf{y} \in \Gamma_1$ and similar for the other interfaces.

At $\Gamma_2$, we have $\normal_1 = -\normal_2$ since these normals have opposing directions (see Fig.~\ref{fig_geometry}). Furthermore, $p(\mathbf{x})$ is continuous across $\Gamma_2$ by the interface condition~\eqref{eq_transmission_pressure}. Hence,
\begin{align*}
    &\int_{\Gamma_2} G_{k_0}(\mathbf{x},\mathbf{y}) \, \alpha_1(\mathbf{y}) p(\mathbf{y}) \normal_1(\mathbf{y}) \dy + \int_{\Gamma_2} G_{k_0}(\mathbf{x},\mathbf{y}) \, \alpha_2(\mathbf{y}) p(\mathbf{y}) \normal_2(\mathbf{y}) \dy \nonumber \\
    &\quad = \int_{\Gamma_2} G_{k_0}(\mathbf{x},\mathbf{y}) \left(\alpha_2(\mathbf{y}) - \alpha_1(\mathbf{y})\right) p(\mathbf{y}) \normal_2(\mathbf{y}) \dy.
\end{align*}
Substitution of the integration by parts in Eq.~\eqref{eq_vie_convolution} yields Eq.~\eqref{eq_vsie_divergence}.
\end{proof}

Even though the VSIE~\eqref{eq_vsie_divergence} has no derivatives acting on the unknown pressure field anymore, gradients and divergences are still present. Let us eliminate the second-order derivatives by applying the exterior Helmholtz equation.

\begin{lemma}
\label{lemma_divergence}
The VSIE~\eqref{eq_vsie_divergence} is equivalent to
\begin{align}
    \frac{\rho_0}{\rho(\mathbf{x})} p(\mathbf{x}) &- \int_{\Omega_1 \cup \Omega_2} G_{k_0}(\mathbf{x},\mathbf{y}) \left(\beta(\mathbf{y}) - k_0^2 \alpha(\mathbf{y})\right) p(\mathbf{y}) \dy \nonumber \\
    &+ \nabla_{\mathbf{x}} \cdot \int_{\Omega_1 \cup \Omega_2} G_{k_0}(\mathbf{x},\mathbf{y}) \left( \nabla_{\mathbf{y}} \alpha(\mathbf{y})\right) p(\mathbf{y}) \dy \nonumber \\
    &+ \int_{\Gamma_1} \left(\frac{\partial}{\partial \normal_1(\mathbf{y})} G_{k_0}(\mathbf{x},\mathbf{y})\right) \alpha_1(\mathbf{y}) p(\mathbf{y}) \dy \nonumber \\
    &+ \int_{\Gamma_2} \left(\frac{\partial}{\partial \normal_2(\mathbf{y})} G_{k_0}(\mathbf{x},\mathbf{y})\right) \left(\alpha_2(\mathbf{y}) - \alpha_1(\mathbf{y})\right) p(\mathbf{y}) \dy \nonumber \\
    &= p_\mathrm{inc}(\mathbf{x}),
    &&\mathbf{x} \in \Omega.
    \label{eq_vsie_helmholtz}
\end{align}
\end{lemma}
\begin{proof}
The Helmholtz Green's function~\eqref{eq_greens_function} satisfies $G_{k_0}(\mathbf{x}, \mathbf{y}) = G(\lvert\mathbf{x} - \mathbf{y}\rvert)$. Hence, it is translation invariant, symmetric, and $\nabla_\mathbf{x} G_{k_0}(\mathbf{x}, \mathbf{y}) = -\nabla_\mathbf{y} G_{k_0}(\mathbf{x}, \mathbf{y})$.

The surface integral in Eq.~\eqref{eq_vsie_divergence} can be written as
\begin{align}
    &\nabla_{\mathbf{x}} \cdot \int_{\Gamma_2} G_{k_0}(\mathbf{x},\mathbf{y}) \left(\alpha_1(\mathbf{y}) - \alpha_2(\mathbf{y})\right) p(\mathbf{y}) \normal_1(\mathbf{y}) \dy \nonumber \\
    &\quad = \int_{\Gamma_2} \left(\nabla_{\mathbf{x}} G_{k_0}(\mathbf{x},\mathbf{y}) \cdot \normal_1(\mathbf{y})\right) \left(\alpha_1(\mathbf{y}) - \alpha_2(\mathbf{y})\right) p(\mathbf{y}) \dy \nonumber \\
    &\quad = -\int_{\Gamma_2} \left(\frac{\partial}{\partial \normal_1(\mathbf{y})} G_{k_0}(\mathbf{x},\mathbf{y})\right) \left(\alpha_1(\mathbf{y}) - \alpha_2(\mathbf{y})\right) p(\mathbf{y}) \dy
    \label{eq_divergence_surface}
\end{align}
with the symmetry of Green's function and standard calculus identities. The integral over interface~$\Gamma_1$ can be rewritten similarly. This expression is easier to evaluate within a Galerkin discretization since the gradient of Green's function has an analytical expression. In fact, one recognizes the well-known double-layer boundary integral operator.

Concerning the divergence of the volume integral in Eq.~\eqref{eq_vsie_divergence}, let us perform the following manipulations:
\begin{align}
    &\nabla_{\mathbf{x}} \cdot \int_{\Omega_1 \cup \Omega_2} \nabla_{\mathbf{y}} \left(G_{k_0}(\mathbf{x},\mathbf{y}) \, \alpha(\mathbf{y})\right) p(\mathbf{y}) \dy \nonumber \\
    &= \nabla_{\mathbf{x}} \cdot \int_{\Omega_1 \cup \Omega_2} \left(\nabla_{\mathbf{y}} G_{k_0}(\mathbf{x},\mathbf{y})\right) \alpha(\mathbf{y}) p(\mathbf{y}) \dy + \nabla_{\mathbf{x}} \cdot \int_{\Omega_1 \cup \Omega_2} G_{k_0}(\mathbf{x},\mathbf{y}) \left( \nabla_{\mathbf{y}} \alpha(\mathbf{y})\right) p(\mathbf{y}) \dy \nonumber \\
    &= \nabla_{\mathbf{x}} \cdot \int_{\Omega_1 \cup \Omega_2} -\nabla_{\mathbf{x}} G_{k_0}(\mathbf{x},\mathbf{y}) \alpha(\mathbf{y}) p(\mathbf{y}) \dy + \nabla_{\mathbf{x}} \cdot \int_{\Omega_1 \cup \Omega_2} G_{k_0}(\mathbf{x},\mathbf{y}) \left( \nabla_{\mathbf{y}} \alpha(\mathbf{y})\right) p(\mathbf{y}) \dy \nonumber \\
    &= -\nabla_{\mathbf{x}}^2 \int_{\Omega_1 \cup \Omega_2} G_{k_0}(\mathbf{x},\mathbf{y}) \alpha(\mathbf{y}) p(\mathbf{y}) \dy + \nabla_{\mathbf{x}} \cdot \int_{\Omega_1 \cup \Omega_2} G_{k_0}(\mathbf{x},\mathbf{y}) \left( \nabla_{\mathbf{y}} \alpha(\mathbf{y})\right) p(\mathbf{y}) \dy
    \label{eq_divergence_volume}
\end{align}
with the product rule and the symmetry of Green's function. The term with the Laplacian~$\nabla^2$ can be simplified by considering the Green's function of the exterior Helmholtz equation~\eqref{eq_helmholtz_dirac}, which satisfies
\begin{align*}
    -(\nabla_{\mathbf{x}}^2 + k_0^2) G_{k_0}(\mathbf{x}, \mathbf{y}) = \delta(\mathbf{x} - \mathbf{y}) \text{ for } \mathbf{x}, \mathbf{y} \in \mathbb{R}^3.
\end{align*}
Integrating against $\alpha(\mathbf{y}) p(\mathbf{y})$ gives
\begin{align*}
    -\int_\Omega (\nabla_{\mathbf{x}}^2 + k_0^2) G_{k_0}(\mathbf{x}, \mathbf{y}) \alpha(\mathbf{y}) p(\mathbf{y}) \dy &= \int_\Omega \delta(\mathbf{x} - \mathbf{y})  \alpha(\mathbf{y}) p(\mathbf{y}) \dy, \\
    -(\nabla_{\mathbf{x}}^2 + k_0^2) \int_\Omega G_{k_0}(\mathbf{x}, \mathbf{y}) \alpha(\mathbf{y}) p(\mathbf{y}) \dy &= \alpha(\mathbf{x}) p(\mathbf{x}).
\end{align*}
Since $\alpha(\mathbf{y})$ has compact support in the interior subdomains, we can write
\begin{equation}
    -\nabla_{\mathbf{x}}^2 \int_{\Omega_1 \cup \Omega_2} G_{k_0}(\mathbf{x},\mathbf{y}) \alpha(\mathbf{y}) p(\mathbf{y}) \dy = k_0^2 \int_{\Omega_1 \cup \Omega_2} G_{k_0}(\mathbf{x},\mathbf{y}) \alpha(\mathbf{y}) p(\mathbf{y}) \dy + \alpha(\mathbf{x}) p(\mathbf{x}),
    \label{eq_laplacian_volume}
\end{equation}
a simpler expression for the Laplacian of the volume integral that is devoid of derivatives.

Substituting expressions~\eqref{eq_divergence_surface}, \eqref{eq_divergence_volume} and~\eqref{eq_laplacian_volume} into the VSIE~\eqref{eq_vsie_divergence} yields
\begin{align}
    &p(\mathbf{x}) - \int_{\Omega_1 \cup \Omega_2} G_{k_0}(\mathbf{x},\mathbf{y}) \beta(\mathbf{y}) p(\mathbf{y}) \dy + k_0^2 \int_{\Omega_1 \cup \Omega_2} G_{k_0}(\mathbf{x},\mathbf{y}) \alpha(\mathbf{y}) p(\mathbf{y}) \dy \nonumber \\
    &\quad + \alpha(\mathbf{x}) p(\mathbf{x}) + \nabla_{\mathbf{x}} \cdot \int_{\Omega_1 \cup \Omega_2} G_{k_0}(\mathbf{x},\mathbf{y}) \left( \nabla_{\mathbf{y}} \alpha(\mathbf{y})\right) p(\mathbf{y}) \dy \nonumber \\
    &\quad+ \int_{\Gamma_1} \left(\frac{\partial}{\partial \normal_1(\mathbf{y})} G_{k_0}(\mathbf{x},\mathbf{y})\right) \alpha_1(\mathbf{y}) p(\mathbf{y}) \dy \nonumber \\
    &\quad + \int_{\Gamma_2} \left(\frac{\partial}{\partial \normal_1(\mathbf{y})} G_{k_0}(\mathbf{x},\mathbf{y})\right) \left(\alpha_1(\mathbf{y}) - \alpha_2(\mathbf{y})\right) p(\mathbf{y}) \dy
    = p_\mathrm{inc}(\mathbf{x})
\end{align}
for $\mathbf{x} \in \Omega$. Notice that $1 + \alpha(\mathbf{x}) = \rho_0/\rho(\mathbf{x})$ by Eq.~\eqref{eq_alpha}. Joining similar integrals simplifies the formulation into Eq.~\eqref{eq_vsie_helmholtz}.
\end{proof}

\subsubsection{VSIE formulation}

To conclude the above manipulations on the VSIE, we state the following theorem.

\begin{theorem}
\label{theorem_vsie_omega}
The solution $p(\mathbf{x})$ of the heterogeneous Helmholtz transmission system~\eqref{eq_helmholtz_exterior}--\eqref{eq_sommerfeld} is implicitly given by the solution of the VSIE
\begin{align}
    \frac{\rho_0}{\rho(\mathbf{x})} p(\mathbf{x}) &- \int_{\Omega_1 \cup \Omega_2} G_{k_0}(\mathbf{x},\mathbf{y}) \left(\beta(\mathbf{y}) - k_0^2 \alpha(\mathbf{y})\right) p(\mathbf{y}) \dy \nonumber \\
    &+ \nabla_{\mathbf{x}} \cdot \int_{\Omega_1 \cup \Omega_2} G_{k_0}(\mathbf{x},\mathbf{y}) \left( \nabla_{\mathbf{y}} \alpha(\mathbf{y})\right) p(\mathbf{y}) \dy \nonumber \\
    &+ \int_{\Gamma_1} \left(\frac{\partial}{\partial \normal_1(\mathbf{y})} G_{k_0}(\mathbf{x},\mathbf{y})\right) \alpha_1(\mathbf{y}) p(\mathbf{y}) \dy \nonumber \\
    &+ \int_{\Gamma_2} \left(\frac{\partial}{\partial \normal_2(\mathbf{y})} G_{k_0}(\mathbf{x},\mathbf{y})\right) \left(\alpha_2(\mathbf{y}) - \alpha_1(\mathbf{y})\right) p(\mathbf{y}) \dy \nonumber \\
    &= p_\mathrm{inc}(\mathbf{x}),
    \label{eq_vsie_omega}
\end{align}
for $\mathbf{x} \in \Omega$.
\end{theorem}
\begin{proof}
This is a direct consequence of Theorems~\ref{theorem_convolution}--\ref{lemma_divergence}.
\end{proof}

Notice that the VSIE~\eqref{eq_vsie_omega} includes common operators for integral equations like the single-layer volume integral, adjoint double-layer volume integral, and double-layer boundary integral operator. However, unlike standard formulations, the integrands also involve heterogeneous material parameters.

Remember that the main goal of the algebraic manipulations was to reduce the order of differentiation. Indeed, Eq.~\eqref{eq_vsie_omega} only has first-order derivatives of the Green's function and the material parameter $\alpha(\mathbf{x})$. By successfully shifting all the derivatives away from the pressure $p(\mathbf{x})$, we have established a representation that supports the extraction of boundary limits.

\subsubsection{VSIE at material interfaces}

There is an important subtlety about the VSIE~\eqref{eq_vsie_omega} in that it is defined for $\Omega = \Omega_0 \cup \Omega_1 \cup \Omega_2$, the union of open subdomains. Hence, it's not defined on the interfaces $\Gamma_1$ or $\Gamma_2$. At the same time, the surface integrals must evaluate the pressure field $p(\mathbf{x})$ on $\Gamma_1$ and $\Gamma_2$. The surface pressure is well-defined, as the interface condition~\eqref{eq_transmission_pressure} requires the pressure field $p(\mathbf{x})$ to be continuous across material interfaces. However, the problem is that the VSIE does not provide an equation to calculate this surface pressure. To obtain a closed system of equations, we will extend the VSIE's domain to the interfaces by taking limit values, i.e., the Dirichlet traces of the VSIE towards the interfaces.

\begin{lemma}
\label{theorem_vsie_gamma}
The Dirichlet traces of the VSIE~\eqref{eq_vsie} are given by
\begin{align}
    \frac{1}{2}\left(\frac{\rho_0}{\rho_1(\mathbf{x})} + 1\right) p(\mathbf{x}) &- \int_{\Omega_1 \cup \Omega_2} G_{k_0}(\mathbf{x},\mathbf{y}) \left(\beta(\mathbf{y}) - k_0^2 \alpha(\mathbf{y})\right) p(\mathbf{y}) \dy \nonumber \\
    &+ \nabla_{\mathbf{x}} \cdot \int_{\Omega_1 \cup \Omega_2} G_{k_0}(\mathbf{x},\mathbf{y}) \left( \nabla_{\mathbf{y}} \alpha(\mathbf{y})\right) p(\mathbf{y}) \dy \nonumber \\
    &+ \int_{\Gamma_1} \left(\frac{\partial}{\partial \normal_1(\mathbf{y})} G_{k_0}(\mathbf{x},\mathbf{y})\right) \alpha_1(\mathbf{y}) p(\mathbf{y}) \dy \nonumber \\
    &+\int_{\Gamma_2} \left(\frac{\partial}{\partial \normal_2(\mathbf{y})} G_{k_0}(\mathbf{x},\mathbf{y})\right) \left(\alpha_2(\mathbf{y}) - \alpha_1(\mathbf{y})\right) p(\mathbf{y}) \dy \nonumber \\
    &= p_\mathrm{inc}(\mathbf{x})
    \label{eq_vsie_g1}
\end{align}
for $\mathbf{x} \in \Gamma_1$ and
\begin{align}
    \frac{1}{2}\left(\frac{\rho_0}{\rho_1(\mathbf{x})} + \frac{\rho_0}{\rho_2(\mathbf{x})}\right) p(\mathbf{x}) &- \int_{\Omega_1 \cup \Omega_2} G_{k_0}(\mathbf{x},\mathbf{y}) \left(\beta(\mathbf{y}) - k_0^2 \alpha(\mathbf{y})\right) p(\mathbf{y}) \dy \nonumber \\
    &+ \nabla_{\mathbf{x}} \cdot \int_{\Omega_1 \cup \Omega_2} G_{k_0}(\mathbf{x},\mathbf{y}) \left( \nabla_{\mathbf{y}} \alpha(\mathbf{y})\right) p(\mathbf{y}) \dy \nonumber \\
    &+ \int_{\Gamma_1} \left(\frac{\partial}{\partial \normal_1(\mathbf{y})} G_{k_0}(\mathbf{x},\mathbf{y})\right) \alpha_1(\mathbf{y}) p(\mathbf{y}) \dy \nonumber \\
    &+\int_{\Gamma_2} \left(\frac{\partial}{\partial \normal_2(\mathbf{y})} G_{k_0}(\mathbf{x},\mathbf{y})\right) \left(\alpha_2(\mathbf{y}) - \alpha_1(\mathbf{y})\right) p(\mathbf{y}) \dy \nonumber \\
    &= p_\mathrm{inc}(\mathbf{x})
    \label{eq_vsie_g2}
\end{align}
for $\mathbf{x} \in \Gamma_2$.
\end{lemma}
\begin{proof}
This is a direct consequence of the jump relations of the double-layer boundary integral operator (e.g.,~\cite[Th.~3.3.14]{sauter2011boundary}). Precisely,
\begin{align}
    &\trace_{\mathrm{D},\mathbf{x}}^- \int_{\Gamma_1} \left(\frac{\partial}{\partial \normal_1(\mathbf{y})} G_{k_0}(\mathbf{x},\mathbf{y})\right) \alpha_1(\mathbf{y}) p(\mathbf{y}) \dy \nonumber \\
    &\quad = -\frac12 \alpha_1(\mathbf{x}) p(\mathbf{x}) + \int_{\Gamma_1} \left(\frac{\partial}{\partial \normal_1(\mathbf{y})} G_{k_0}(\mathbf{x},\mathbf{y})\right) \alpha_1(\mathbf{y}) p(\mathbf{y}) \dy, \quad \mathbf{x} \in \Gamma_1,
    \label{eq_jump_g1}
\end{align}
and
\begin{align}
    &\trace_{\mathrm{D},\mathbf{x}}^- \int_{\Gamma_2} \left(\frac{\partial}{\partial \normal_2(\mathbf{y})} G_{k_0}(\mathbf{x},\mathbf{y})\right) \left(\alpha_2(\mathbf{y}) - \alpha_1(\mathbf{y})\right) p(\mathbf{y}) \dy \nonumber \\
    &\quad = -\frac12 \left(\alpha_2(\mathbf{y}) - \alpha_1(\mathbf{y})\right) p(\mathbf{x}) \nonumber \\
    &\quad + \int_{\Gamma_2} \left(\frac{\partial}{\partial \normal_2(\mathbf{y})} G_{k_0}(\mathbf{x},\mathbf{y})\right) \left(\alpha_2(\mathbf{y}) - \alpha_1(\mathbf{y})\right) p(\mathbf{y}) \dy, \quad \mathbf{x} \in \Gamma_2.
    \label{eq_jump_g2}
\end{align}
Hence, taking the interior Dirichlet trace from $\Omega_1$ to $\Gamma_1$ of the VSIE~\eqref{eq_vsie} yields
\begin{align}
    \frac{\rho_0}{\rho_1(\mathbf{x})} p(\mathbf{x}) &- \int_{\Omega_1 \cup \Omega_2} G_{k_0}(\mathbf{x},\mathbf{y}) \left(\beta(\mathbf{y}) - k_0^2 \alpha(\mathbf{y})\right) p(\mathbf{y}) \dy \nonumber \\
    &+ \nabla_{\mathbf{x}} \cdot \int_{\Omega_1 \cup \Omega_2} G_{k_0}(\mathbf{x},\mathbf{y}) \left( \nabla_{\mathbf{y}} \alpha(\mathbf{y})\right) p(\mathbf{y}) \dy \nonumber \\
    &- \frac12 \alpha_1(\mathbf{x}) p(\mathbf{x}) \nonumber\\
    &+ \int_{\Gamma_1} \left(\frac{\partial}{\partial \normal_1(\mathbf{y})} G_{k_0}(\mathbf{x},\mathbf{y})\right) \alpha_1(\mathbf{y}) p(\mathbf{y}) \dy \nonumber \\
    &+ \int_{\Gamma_2} \left(\frac{\partial}{\partial \normal_2(\mathbf{y})} G_{k_0}(\mathbf{x},\mathbf{y})\right) \left(\alpha_2(\mathbf{y}) - \alpha_1(\mathbf{y})\right) p(\mathbf{y}) \dy \nonumber \\
    &= p_\mathrm{inc}(\mathbf{x}),
    &&\mathbf{x} \in \Gamma_1,
    \label{eq_jump_vsie_g1}
\end{align}
by the jump relation~\eqref{eq_jump_g1} and the fact that the limits of the other integrals can be evaluated directly. The definition of $\alpha(\mathbf{x})$ provides
\begin{align*}
    \frac{\rho_0}{\rho_1(\mathbf{x})} - \frac12 \alpha_1(\mathbf{x})
    &= \frac{\rho_0}{\rho_1(\mathbf{x})} - \frac12 \left(\frac{\rho_0}{\rho_1(\mathbf{x})} - 1\right)
    = \frac12 \left(\frac{\rho_0}{\rho_1(\mathbf{x})} + 1\right),
\end{align*}
and substitution in Eq.~\eqref{eq_jump_vsie_g1} yields Eq.~\eqref{eq_vsie_g1}. Notice that taking the exterior Dirichlet trace from $\Omega_0$ to $\Gamma_1$ will result in the same equation. Now, taking the interior Dirichlet trace from $\Omega_2$ to $\Gamma_2$ of the VSIE~\eqref{eq_vsie} yields
\begin{align}
    \frac{\rho_0}{\rho_2(\mathbf{x})} p(\mathbf{x}) &- \int_{\Omega_1 \cup \Omega_2} G_{k_0}(\mathbf{x},\mathbf{y}) \left(\beta(\mathbf{y}) - k_0^2 \alpha(\mathbf{y})\right) p(\mathbf{y}) \dy \nonumber \\
    &+ \nabla_{\mathbf{x}} \cdot \int_{\Omega_1 \cup \Omega_2} G_{k_0}(\mathbf{x},\mathbf{y}) \left( \nabla_{\mathbf{y}} \alpha(\mathbf{y})\right) p(\mathbf{y}) \dy \nonumber \\
    &+ \int_{\Gamma_1} \left(\frac{\partial}{\partial \normal_1(\mathbf{y})} G_{k_0}(\mathbf{x},\mathbf{y})\right) \alpha_1(\mathbf{y}) p(\mathbf{y}) \dy \nonumber \\
    &- \frac12 (\alpha_2(\mathbf{x}) - \alpha_1(\mathbf{x})) p(\mathbf{x}) \nonumber\\
    &+ \int_{\Gamma_2} \left(\frac{\partial}{\partial \normal_2(\mathbf{y})} G_{k_0}(\mathbf{x},\mathbf{y})\right) \left(\alpha_2(\mathbf{y}) - \alpha_1(\mathbf{y})\right) p(\mathbf{y}) \dy \nonumber \\
    &= p_\mathrm{inc}(\mathbf{x}),
    &&\mathbf{x} \in \Gamma_2,
    \label{eq_jump_vsie_g2}
\end{align}
by the jump relation~\eqref{eq_jump_g2}. Also,
\begin{align*}
    \frac{\rho_0}{\rho_2(\mathbf{x})} - \frac12 (\alpha_2(\mathbf{x}) - \alpha_1(\mathbf{x}))
    &= \frac{\rho_0}{\rho_2(\mathbf{x})} - \frac12 \left(\frac{\rho_0}{\rho_2(\mathbf{x})} - 1\right) + \frac12 \left(\frac{\rho_0}{\rho_1(\mathbf{x})} - 1\right) \\
    &= \frac12 \left(\frac{\rho_0}{\rho_2(\mathbf{x})} + \frac{\rho_0}{\rho_1(\mathbf{x})}\right),
\end{align*}
and substitution in Eq.~\eqref{eq_jump_vsie_g2} yields Eq.~\eqref{eq_vsie_g2}. Notice that, again, taking the exterior Dirichlet trace from $\Omega_1$ to $\Gamma_2$ will result in the same equation.
\end{proof}

Given the similarities between the VSIEs, let us define the mass function
\begin{equation}
    m(\mathbf{x}) = 
    \begin{cases}
        1 & \text{for } \mathbf{x} \in \Omega_0, \\
        \frac{\rho_0}{\rho_1(\mathbf{x})} & \text{for } \mathbf{x} \in \Omega_1, \\
        \frac{\rho_0}{\rho_2(\mathbf{x})} & \text{for } \mathbf{x} \in \Omega_2, \\
        \frac{1}{2}\left(\frac{\rho_0}{\rho_1(\mathbf{x})}+1\right) & \text{for } \mathbf{x} \in \Gamma_1, \\
        \frac{1}{2}\left(\frac{\rho_0}{\rho_1(\mathbf{x})} + \frac{\rho_0}{\rho_2(\mathbf{x})}\right)& \text{for } \mathbf{x} \in \Gamma_2.
    \end{cases}
    \label{eq_mass_function}
\end{equation}
Notice that this expression can be interpreted as the local average of the density contrast $\rho_0/\rho(\mathbf{x})$. This definition simplifies the final VSIE formulation.

\begin{theorem}
\label{theorem_vsie_final}
The solution $p(\mathbf{x})$ of the heterogeneous Helmholtz transmission system~\eqref{eq_helmholtz_exterior}--\eqref{eq_sommerfeld} is implicitly given by the solution of the VSIE
\begin{align}
    m(\mathbf{x}) p(\mathbf{x}) &- \int_{\Omega_1 \cup \Omega_2} G_{k_0}(\mathbf{x},\mathbf{y}) \left(\beta(\mathbf{y}) - k_0^2 \alpha(\mathbf{y})\right) p(\mathbf{y}) \dy \nonumber \\
    &+ \nabla_{\mathbf{x}} \cdot \int_{\Omega_1 \cup \Omega_2} G_{k_0}(\mathbf{x},\mathbf{y}) \left( \nabla_{\mathbf{y}} \alpha(\mathbf{y})\right) p(\mathbf{y}) \dy \nonumber \\
    &+ \int_{\Gamma_1} \left(\frac{\partial}{\partial \normal_1(\mathbf{y})} G_{k_0}(\mathbf{x},\mathbf{y})\right) \alpha_1(\mathbf{y}) p(\mathbf{y}) \dy \nonumber \\
    &+ \int_{\Gamma_2} \left(\frac{\partial}{\partial \normal_2(\mathbf{y})} G_{k_0}(\mathbf{x},\mathbf{y})\right) \left(\alpha_2(\mathbf{y}) - \alpha_1(\mathbf{y})\right) p(\mathbf{y}) \dy \nonumber \\
    &= p_\mathrm{inc}(\mathbf{x}),
    \label{eq_vsie}
\end{align}
for $\mathbf{x} \in \Rthree$.
\end{theorem}
\begin{proof}
This is a direct consequence of substituting Eq.~\eqref{eq_mass_function} into Eqns~\eqref{eq_vsie_omega}, \eqref{eq_vsie_g1} and~\eqref{eq_vsie_g2} in Theorems~\ref{theorem_vsie_omega} and~\ref{theorem_vsie_gamma}.
\end{proof}

\subsubsection{Special cases}

Let us consider several relevant special cases in which the VSIE~\eqref{eq_vsie} simplifies to shorter expressions.

\begin{remark}
In the case of a single domain~$\Omega_1$, we can eliminate $\Omega_2$ from the geometry and follow the same derivation to reach
\begin{align}
    &m(\mathbf{x}) p(\mathbf{x}) - \int_{\Omega_1} G_{k_0}(\mathbf{x},\mathbf{y}) \left(\beta(\mathbf{y}) - k_0^2 \alpha(\mathbf{y})\right) p(\mathbf{y}) \dy \nonumber \\
    &\quad + \nabla_{\mathbf{x}} \cdot \int_{\Omega_1} G_{k_0}(\mathbf{x},\mathbf{y}) \left( \nabla_{\mathbf{y}} \alpha(\mathbf{y})\right) p(\mathbf{y}) \dy \nonumber \\
    &\quad + \int_{\Gamma_1} \left(\frac{\partial}{\partial \normal_1(\mathbf{y})} G_{k_0}(\mathbf{x},\mathbf{y})\right) \alpha_1(\mathbf{y}) p(\mathbf{y}) \dy
    = p_\mathrm{inc}(\mathbf{x})
    \label{eq_vsie_single_domain}
\end{align}
for $\mathbf{x} \in \Rthree$.
\end{remark}

\begin{remark}
For a globally continuous density $\rho(\mathbf{x}) \in \mathcal{C}(\mathbb{R}^3)$, we have $\alpha_1(\mathbf{x}) = 0$ for $\mathbf{x} \in \Gamma_1$ and $\alpha_2(\mathbf{x}) = \alpha_1(\mathbf{x})$ for $\mathbf{x} \in \Gamma_2$. Then, the surface integrals vanish, and the VSIE~\eqref{eq_vsie} simplifies to the VIE given by
\begin{align}
    &m(\mathbf{x}) p(\mathbf{x}) - \int_{\Omega_1 \cup \Omega_2} G_{k_0}(\mathbf{x},\mathbf{y}) \left(\beta(\mathbf{y}) - k_0^2 \alpha(\mathbf{y})\right) p(\mathbf{y}) \dy \nonumber \\
    &\quad + \nabla_{\mathbf{x}} \cdot \int_{\Omega_1 \cup \Omega_2} G_{k_0}(\mathbf{x},\mathbf{y}) \left( \nabla_{\mathbf{y}} \alpha(\mathbf{y})\right) p(\mathbf{y}) \dy
    = p_\mathrm{inc}(\mathbf{x})
\end{align}
for $\mathbf{x} \in \Rthree$.
\end{remark}

\begin{remark}
When the density is constant everywhere, i.e., $\rho(\mathbf{x}) = \rho_0$ for all $\mathbf{x} \in \Rthree$, we have $\alpha(\mathbf{x}) = 0$ and the VSIE~\eqref{eq_vsie} simplifies to
\begin{align}
    p(\mathbf{x}) - \int_{\Omega_1 \cup \Omega_2} G_{k_0}(\mathbf{x},\mathbf{y}) \beta(\mathbf{y}) p(\mathbf{y}) \dy = p_\mathrm{inc}(\mathbf{x}), \quad \mathbf{x} \in \Rthree.
    \label{eq_lippmann_schwinger}
\end{align}
This is the well-studied Lippmann-Schwinger equation~\cite{kirsch2009operator}.
\end{remark}

\subsection{Numerical discretization}
\label{sec_galerkin}

We choose to write the integro-differential equation~\eqref{eq_vsie} in its weak form, which will allow solving the VSIE with Galerkin or collocation discretization.

\subsubsection{Computational domain}

The VSIE~\eqref{eq_vsie} must be solved for the full interior domain $\Omega_1 \cup \Omega_2 \cup \Gamma_1 \cup \Gamma_2$, the union of the open subdomains and material interfaces. After obtaining the interior pressure, the exterior field $p(\mathbf{x})$ for $\mathbf{x} \in \Omega_0$ can be calculated directly by applying the VSIE as a representation formula to the interior solution. 

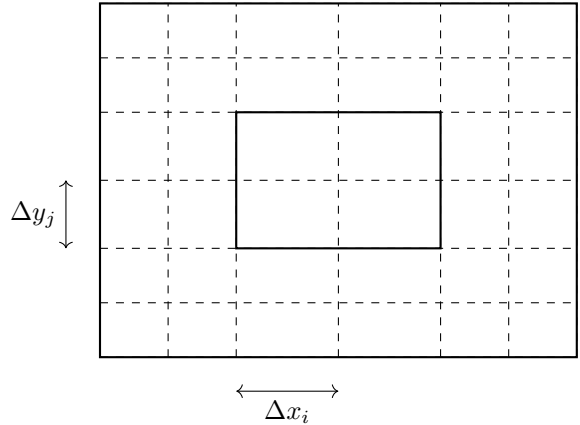
\begin{figure}[htbp]
    \centering
    \begin{tikzpicture}[scale=0.9]
        \def\xsteps{{0,1,2,2.5,4,5.5,7}}
        \def\ysteps{{0,0.8,1.6,2.6,3.6,4.4,5.2}}
        \draw[thick] (0,0) rectangle (7,5.2);
        \draw[thick] (2,1.6) rectangle (5,3.6);
        \foreach \x in {0,1,2,3.5,5,6,7} {
            \draw[dashed] (\x,0) -- (\x,5.2);
        }
        \foreach \y in {0,0.8,1.6,2.6,3.6,4.4,5.2} {
            \draw[dashed] (0,\y) -- (7,\y);
        }
        \draw[<->] (2, -0.5) -- (3.5, -0.5) node[midway,below] {$\Delta x_i$};
        \draw[<->] (-0.5, 1.6) -- (-0.5, 2.6) node[midway,left] {$\Delta y_j$};
    \end{tikzpicture}
    \caption{An example nonuniform cuboid mesh for nested cubes. The figure displays a 2D slice of the 3D mesh (dashed lines), with different element sizes in the two subdomains (bold lines) and cuboid faces that exactly match the interface.}
    \label{fig_grid}
\end{figure}

We use a cuboid mesh in the interior domain. That is, the bounded domains $\Omega_1$ and $\Omega_2$ are subdivided into $N_x \times N_y \times N_z$ rectangular cuboids of sizes $\Delta x_i \times \Delta y_j \times \Delta z_k$. No hanging nodes are allowed, meaning that each cuboid's face connects to the face of its neighboring element. At the interfaces $\Gamma_1$ and $\Gamma_2$, we use rectangular elements. Importantly, we generate the meshes such that the rectangular surface grid matches exactly the faces of the volumetric cuboid grid. To achieve this geometric fit between the surface and volume grids, we use nonuniform cuboid meshes. See Figure~\ref{fig_grid} for a slice of a nonuniform grid and Fig.~\ref{fig_benchmark4_3D} for a three-dimensional visualization. Finally, we emphasize that the weak formulation and numerical discretization can readily be applied to other mesh types, such as tetrahedra, but are left for future implementations. Notice that cuboid grids, also called voxel meshes, are commonly used in computational engineering, for example, to store biomedical image data~\cite{almuna2026full}.

\subsubsection{Weak formulation}

Let us define $\bar\Omega = \Omega_1 \cup \Omega_2 \cup \Gamma_1 \cup \Gamma_2$, the union of the bounded, open subdomains and the interfaces. Testing the VSIE~\eqref{eq_vsie} over the entire interior domain~$\bar\Omega$ against a function $\psi$ yields
\begin{align}
    &\int_{\bar\Omega} m(\mathbf{x}) p(\mathbf{x}) \psi(\mathbf{x}) \dx \nonumber \\
    &\quad - \int_{\bar\Omega} \int_{\Omega_1 \cup \Omega_2} G_{k_0}(\mathbf{x},\mathbf{y}) \left(\beta(\mathbf{y}) - k_0^2 \alpha(\mathbf{y})\right) p(\mathbf{y}) \dy \, \psi(\mathbf{x}) \dx \nonumber \\
    &\quad + \int_{\bar\Omega} \nabla_{\mathbf{x}} \cdot \int_{\Omega_1 \cup \Omega_2} G_{k_0}(\mathbf{x},\mathbf{y}) \left( \nabla_{\mathbf{y}} \alpha(\mathbf{y})\right) p(\mathbf{y}) \dy \, \psi(\mathbf{x}) \dx \nonumber \\
    &\quad+ \int_{\bar\Omega}\int_{\Gamma_1} \left(\frac{\partial}{\partial \normal_1(\mathbf{y})} G_{k_0}(\mathbf{x},\mathbf{y})\right) \alpha_1(\mathbf{y}) p(\mathbf{y}) \dy \, \psi(\mathbf{x})\dx \nonumber \\
    &\quad + \int_{\bar\Omega} \int_{\Gamma_2} \left(\frac{\partial}{\partial \normal_2(\mathbf{y})} G_{k_0}(\mathbf{x},\mathbf{y})\right) \left(\alpha_2(\mathbf{y}) - \alpha_1(\mathbf{y})\right) p(\mathbf{y}) \dy \, \psi(\mathbf{x}) \dx \nonumber \\
    &\quad = \int_{\bar\Omega} p_\mathrm{inc}(\mathbf{x}) \psi(\mathbf{x}) \dx,
    \label{eq_weak_form_global}
\end{align}
where $\dx$ denotes integration with respect to the standard volume measure for $\mathbf{x}\in\Omega_1\cup\Omega_2$ and the surface measure for $\mathbf{x}\in\Gamma_1\cup\Gamma_2$. Notice that, as $\bar\Omega = \Omega_1 \cup \Omega_2 \cup \Gamma_1 \cup \Gamma_2$, testing functions have to be defined on the open subdomains and the interfaces.

As usual, the pressure field $p(\mathbf{x})$ was assumed to be in $H^1(\Omega)$. This allowed us to obtain the VSIE~\eqref{eq_vsie} for the solution of the Helmholtz equation~\eqref{eq_perturbed_helmholtz}. However, imposing $H^1$~conformity on the discrete spaces is computationally demanding as it requires, at least, linear function spaces. At the same time, the VSIE~\eqref{eq_vsie} is valid for pressure fields in the larger $L^2$ function space. We exploit this opportunity and design a discretization scheme amenable to piecewise constant function spaces. 

We will search for a discrete solution $p_{L^2}(\mathbf{x})$ in $\bar\Omega$ such that its restriction to each subdomain is in $L^2$. Precisely, $p_{L^2}|_{\Omega_i} \in L^2(\Omega_i)$ and $p_{L^2}|_{\Gamma_i} \in L^2(\Gamma_i)$ for $i = 1, 2$. This approach can also be interpreted as searching for a set
\begin{align*}
    (p_{\Omega_1}, p_{\Omega_2}, p_{\Gamma_1}, p_{\Gamma_2}) \in L^2(\Omega_1) \times L^2(\Omega_2) \times L^2(\Gamma_1) \times L^2(\Gamma_2)
\end{align*}
of four independent functions, each extended by zero outside their respective domains. The discrete solution can then be reconstructed as
\begin{equation}
    p_{L^2}(\mathbf{x}) = p_{\Omega_1}(\mathbf{x}) + p_{\Omega_2}(\mathbf{x}) + p_{\Gamma_1}(\mathbf{x}) + p_{\Gamma_2}(\mathbf{x}), \quad \mathbf{x} \in \bar\Omega.
    \label{eq_discrete_subdomain_decomposition}
\end{equation}
Hence, we do not enforce the discrete solution to be continuous in the interior domain, nor do we explicitly enforce the transmission condition~\eqref{eq_transmission_pressure} at the material interfaces. In other words, we no longer impose the physical continuity across the interfaces via the functional setting, but rather leverage it via the integral operators in the VSIE. This decoupling allows us to project the continuous integral equations onto discrete function spaces of low-order polynomials.

The discrete solution $p_{L^2}$ will be interpreted as the acoustic field in the interior~$\bar\Omega$. The pressure field in the exterior~$\Omega_0$ can then be calculated using the VSIE~\eqref{eq_vsie} as a representation formula.

\subsubsection{Linear system}

The weak formulation~\eqref{eq_weak_form_global} involves combinations between all four geometric components: the two volumetric regions~$\Omega_1$ and~$\Omega_2$, and the surfaces~$\Gamma_1$ and~$\Gamma_2$. The coupled system reads
\begin{align}
    &\left(\begin{bmatrix}
        \ID_{\Omega_1} & 0 & 0 & 0 \\
        0 & \ID_{\Omega_2} & 0 & 0 \\
        0 & 0 & \ID_{\Gamma_1} & 0 \\
        0 & 0 & 0 & \ID_{\Gamma_2}
    \end{bmatrix}\right. \nonumber \\
    &+
    \left.\begin{bmatrix}
        -\SL_{\Omega_1,\Omega_1} + \AD_{\Omega_1,\Omega_1} & - \SL_{\Omega_1,\Omega_2} + \AD_{\Omega_1,\Omega_2} & \DL_{\Omega_1,\Gamma_1} & \DL_{\Omega_1, \Gamma_2} \\
        -\SL_{\Omega_2,\Omega_1} + \AD_{\Omega_2,\Omega_1} & -\SL_{\Omega_2,\Omega_2} + \AD_{\Omega_2,\Omega_2} & \DL_{\Omega_2,\Gamma_1} & \DL_{\Omega_2, \Gamma_2} \\
        -\SL_{\Gamma_1,\Omega_1} + \AD_{\Gamma_1,\Omega_1} & -\SL_{\Gamma_1,\Omega_2} + \AD_{\Gamma_1,\Omega_2} & \DL_{\Gamma_1,\Gamma_1} & \DL_{\Gamma_1,\Gamma_2} \\
        -\SL_{\Gamma_2,\Omega_1} + \AD_{\Gamma_2,\Omega_1} & -\SL_{\Gamma_2,\Omega_2} + \AD_{\Gamma_2,\Omega_2} & \DL_{\Gamma_2,\Gamma_1} & \DL_{\Gamma_2,\Gamma_2}
    \end{bmatrix}
    \right)
    \begin{bmatrix}
        p_{\Omega_1} \\
        p_{\Omega_2} \\
        p_{\Gamma_1} \\
        p_{\Gamma_2}
    \end{bmatrix}
    =
    \begin{bmatrix}
        {f}_{\Omega_1} \\
        {f}_{\Omega_2} \\
        {f}_{\Gamma_1}\\
        {f}_{\Gamma_2}
    \end{bmatrix}
\end{align}
where
\begin{align}
    &\ID_{\Sigma}[q(\mathbf{x})] = \int_{\Sigma} m(\mathbf{x}) q(\mathbf{x}) \psi(\mathbf{x}) \dx, \\
    &\SL_{\Sigma,\Xi}[q(\mathbf{y})] = \int_{\Sigma} \int_{\Xi} G_{k_0}(\mathbf{x},\mathbf{y}) \left(\beta(\mathbf{y}) - k_0^2 \alpha(\mathbf{y})\right) q(\mathbf{y}) \dy \, \psi(\mathbf{x}) \dx, \\
    &\DL_{\Sigma,\Gamma_1}[q(\mathbf{y})] = \int_{\Sigma} \int_{\Gamma_1} \left(\frac{\partial}{\partial \normal_1(\mathbf{y})} G_{k_0}(\mathbf{x},\mathbf{y})\right) \alpha_1(\mathbf{y}) q(\mathbf{y}) \dy \, \psi(\mathbf{x}) \dx, \\
    &\DL_{\Sigma,\Gamma_2}[q(\mathbf{y})] = \int_{\Sigma} \int_{\Gamma_2} \left(\frac{\partial}{\partial \normal_2(\mathbf{y})} G_{k_0}(\mathbf{x},\mathbf{y})\right) \left(\alpha_2(\mathbf{y}) - \alpha_1(\mathbf{y})\right) q(\mathbf{y}) \dy \, \psi(\mathbf{x}) \dx, \\
    &\AD_{\Sigma,\Xi}[q(\mathbf{y})] = \int_{\Sigma} \nabla_{\mathbf{x}} \cdot \int_{\Xi} G_{k_0}(\mathbf{x},\mathbf{y}) \left( \nabla_{\mathbf{y}} \alpha(\mathbf{y})\right) q(\mathbf{y}) \dy \, \psi(\mathbf{x}) \dx,
\end{align}
the weak form of the scaled mass, single-layer, double-layer, and adjoint double-layer integral operators, respectively, and
\begin{equation}
    f_{\Sigma} = \int_{\Sigma} p_\mathrm{inc}(\mathbf{x}) \psi(\mathbf{x}) \dx
\end{equation}
is the weak formulation of the incident wave field. Here, $\Sigma \in \{\Omega_1, \Omega_2, \Gamma_1, \Gamma_2\}$ any of the integration domains.

\subsubsection{Discrete representation}

We discretize the solution~\eqref{eq_discrete_subdomain_decomposition} as
\begin{equation}
    p_{L^2}(\mathbf{y}) = \sum_{i=1}^{N_{\Omega_1}} c_i \phi_i(\mathbf{y}) + \sum_{i=1}^{N_{\Omega_2}} \hat{c}_i \hat{\phi}_i(\mathbf{y}) + \sum_{i=1}^{N_{\Gamma_1}}\mathring c_i\mathring \phi_i(\mathbf{y}) + \sum_{i=1}^{N_{\Gamma_2}} \tilde{c}_i \tilde{\phi}_i(\mathbf{y})
\end{equation}
for basis functions $\phi_i$, $\hat{\phi}_i$, $\mathring{\phi}_i$, and $\tilde{\phi}_i$ that have local support in $\Omega_1$, $\Omega_2$, $\Gamma_1$, and $\Gamma_2$, respectively. Since no derivatives of the unknown pressure are present in the weak formulation, we choose piecewise-constant (P0) basis functions. Precisely, $\phi_i(\mathbf{y})$ and $\hat{\phi}_i(\mathbf{y})$ equal one inside a single cuboid mesh element in $\Omega_1$ and $\Omega_2$, respectively, and zero everywhere else. The basis functions $\mathring\phi_i(\mathbf{y})$ and $\tilde\phi_i(\mathbf{y})$ are defined on the interface $\Gamma_1$ and $\Gamma_2$, respectively, and equal one on a specific surface element related to a face of the cuboid mesh, and zero anywhere else. We combine these P0 basis functions with a collocation approach. That is, we select the test functions from Dirac delta distributions at the centroid of each cuboid mesh element in $\Omega_1$ and $\Omega_2$ and at the centroid of each rectangular face element on~$\Gamma_1$ and~$\Gamma_2$.

All integrals are evaluated with a midpoint quadrature rule. The weakly-singular integrals for self-interactions are evaluated with analytical expressions on a sphere of equal volume to the grid element (see, e.g.,~\cite{groth2021accelerating}).

\begin{remark}
    While we present the discretization as a point collocation method, we note that it is algebraically equivalent to a $P_0$-$P_0$ Galerkin discretization, where the outer testing integrals are approximated using a midpoint quadrature rule. Furthermore, our discretization approach is reminiscent of the classical direct dipole approximation of the VIE~\cite{draine1994discrete}.
\end{remark}

\subsection{Boundary Element Method}
\label{sec_bem}

In the case of piecewise constant material parameters, the Helmholtz system~\eqref{eq_helmholtz_subdomains} can also be solved with the BEM on the material interfaces. We use the standard PMCHWT formulation for nested domains as a benchmark algorithm. Details on the BEM formulations and discretizations can be found in previous studies that use the same notation and implementation platform, such as~\cite{wout2021benchmarking}.

\subsection{Finite Element Method}
\label{sec_fembem}

For locally heterogeneous materials, we couple the FEM with a BEM at the outer surface. Precisely, we use the standard Johnson-Nédélec coupling technique~\cite{johnson1980coupling}, adjusted to nested domains. Notice that when $\Omega_1$ is homogeneous, one can use a FEM-BEM coupled system using a BEM at both $\Gamma_1$ and $\Gamma_2$, and a FEM in $\Omega_2$. We follow the same approach as our previous study~\cite{wout2022fembem} and refer to that manuscript for details on the FEM-BEM formulation and its implementation.

\section{Results}
\label{sec:results}

To verify our proposed VSIE formulation, we created a benchmark suite that allows comparison with spherical harmonics, BEM, and FEM-BEM algorithms. See Table~\ref{table_benchmarks} for a summary. We choose a relatively simple piecewise constant speed of sound to prioritize assessing the influence of the density function. It is the density that yields the surface integrals in the VSIE, thereby distinguishing it from the commonly used Lippmann-Schwinger equation.

\begin{table}[htbp]
    \caption{The benchmark suite for the acoustic VSIE formulation for nested domains.}
    \label{table_benchmarks}
    \centering
    \begin{tabular}{rlp{40mm}l}
        \toprule
        ID & speed of sound & density & comparison \\
        \midrule
        1 & piecewise constant & constant & analytic solution \\
        2 & piecewise constant & piecewise constant & BEM \\
        3 & piecewise constant & locally heterogeneous and globally continuous & FEM-BEM \\
        4 & piecewise constant & locally heterogeneous with discontinuity & FEM-BEM \\
        \bottomrule
    \end{tabular}
\end{table}

\subsection{Computational settings}

The VSIE~\eqref{eq_vsie} was numerically solved with collocation and piecewise-constant basis functions on cuboid meshes, as explained in Section~\ref{sec_galerkin}. The discrete matrix equation is solved using the iterative Generalized Minimal Residual (GMRES) algorithm with a relative tolerance of $10^{-5}$ and no restarts. No preconditioning was used for the VSIE, while we used mass-matrix preconditioning for the BEM and the sparse LU decomposition of the FEM block as the preconditioner in the FEM-BEM coupled system.

When reporting the number of grid elements per wavelength, we use the largest element across the entire grid. For the element size, we calculate the largest side length in a voxel mesh and the element diameter in a triangular or tetrahedral mesh. Furthermore, the shortest wavelength across the subdomains is used for reporting elements per wavelength.

We use a plane-wave field
\begin{equation}
    p_\mathrm{inc}(\mathbf{x}) = e^{\imath k_0 \mathbf{x} \cdot \hat{\mathbf{d}}}, \quad \mathbf{x} \in \Rthree,
    \label{eq_planewave}
\end{equation}
as the incident wave field~\eqref{eq_pinc}, which is an idealized model for a point source far away from the scatterer. Here, $\hat{\mathbf{d}}$ is a unit-sized vector representing the propagation direction of the plane-wave field. If not stated otherwise, we use $\hat{\mathbf{d}} = \hat{\mathbf{x}}$, the positive $x$-direction. The geometry's shape, material's density, and speed of sound differ across benchmarks and are chosen to represent structures and materials common in biomedical engineering~\cite{haqshenas2021fast}. No attenuation will be used, although the VSIE methodology readily allows for complex-valued wavenumbers.

The VSIE algorithm was implemented in Python using the SciPy library~\cite{scipy} for linear algebra. The matrix assembly was shared-memory parallelized with multi-threading and just-in-time compilation via the Numba library~\cite{numba}. The benchmarks with BEM and FEM-BEM coupling were performed with the open-source libraries Bempp~\cite{smigaj2015solving, betcke2021bempp} and Fenics~\cite{fenicsx}, using piecewise-linear Galerkin discretization on triangular and tetrahedral meshes generated with Gmsh~\cite{geuzaine2009gmsh}.

While the small-scale examples were run on personal computers, we used a dedicated computer node for the convergence studies. The hardware has 2~Intel(R) Xeon(R) CLX 4216 processors, 32 cores, and 2048~GB of RAM.

\subsection{Verification with spherical harmonics}

One of the rare cases in which the Helmholtz equation admits an analytical solution is a spherical domain with constant material parameters. In this case, the pressure field can be expressed as a series expansion in spherical harmonics~\cite{anderson1950sound}. Having an analytical solution allows us to calculate the VSIE's error exactly.

Let us consider a sphere of radius 2.5 as the inner domain $\Omega_2$. The intermediate domain $\Omega_1$ covers the cube~$[-3,3]^3$. The density is taken to be globally constant, normalized as $\rho(\mathbf{x}) = 1$ for $\mathbf{x} \in \Rthree$. The speed of sound is piecewise constant with a moderate contrast: $c_0 = c_1 = 1$ and $c_2 = 5/6$. The frequency is chosen such that $\lambda_0 = 2$ for the exterior wavelength. A uniform cuboid mesh is generated inside the cube~$[-3,3]^3$, which means that the sphere's surface is not exactly represented, causing staircase effects at $\Gamma_2$. Finally, notice that in this special case, the VSIE~\eqref{eq_vsie} reduces to the Lippmann-Schwinger equation~\eqref{eq_lippmann_schwinger}.

\begin{figure}[htbp]
    \centering
    \includegraphics[width=\textwidth]{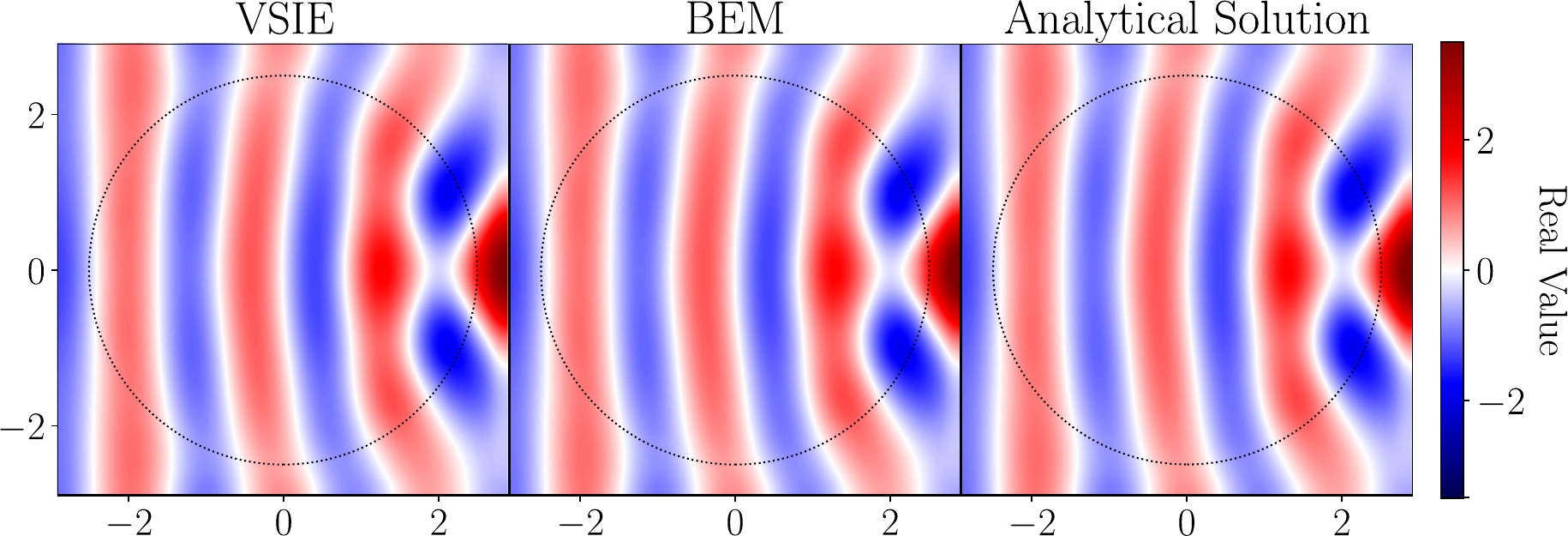}
    \caption{The real part of the acoustic field propagating through a sphere, at a slice through the center ($z=0$). The three panels represent the VSIE, BEM, and analytical solution. The meshes use 8 elements per interior wavelength.}
    \label{fig_sphere_field}
\end{figure}

Figure~\ref{fig_sphere_field} shows the acoustic field propagated through the sphere, on a slice at $z=0$, directly taken from the solution in the cuboid's centers. The numerically computed solutions from the VSIE and BEM algorithms visually match the analytical solution perfectly. More precisely, we can calculate the relative error
\begin{equation}
    E_\mathrm{rel} = \frac{\norm{p_\mathrm{numerical} - p_\mathrm{benchmark}}}{\norm{p_\mathrm{benchmark}}}
    \label{eq_error}
\end{equation}
with the pressure calculated on the visualization slice. Here, we use the analytical solution as the benchmark, and both the Frobenius $\norm{\cdot}_2$ and maximum $\norm{\cdot}_\infty$ norms. For the VSIE, we retrieve the acoustic field directly from the solution vector. For the BEM, we calculate the field at the voxel centers using the representation formula. The relative errors of the VIE and BEM compared with the analytical solution are lower than 3\% in the simulation for Figure~\ref{fig_sphere_field}, which uses 8 elements per wavelength.

\begin{figure}
    \centering
    \includegraphics[width=0.8\linewidth]{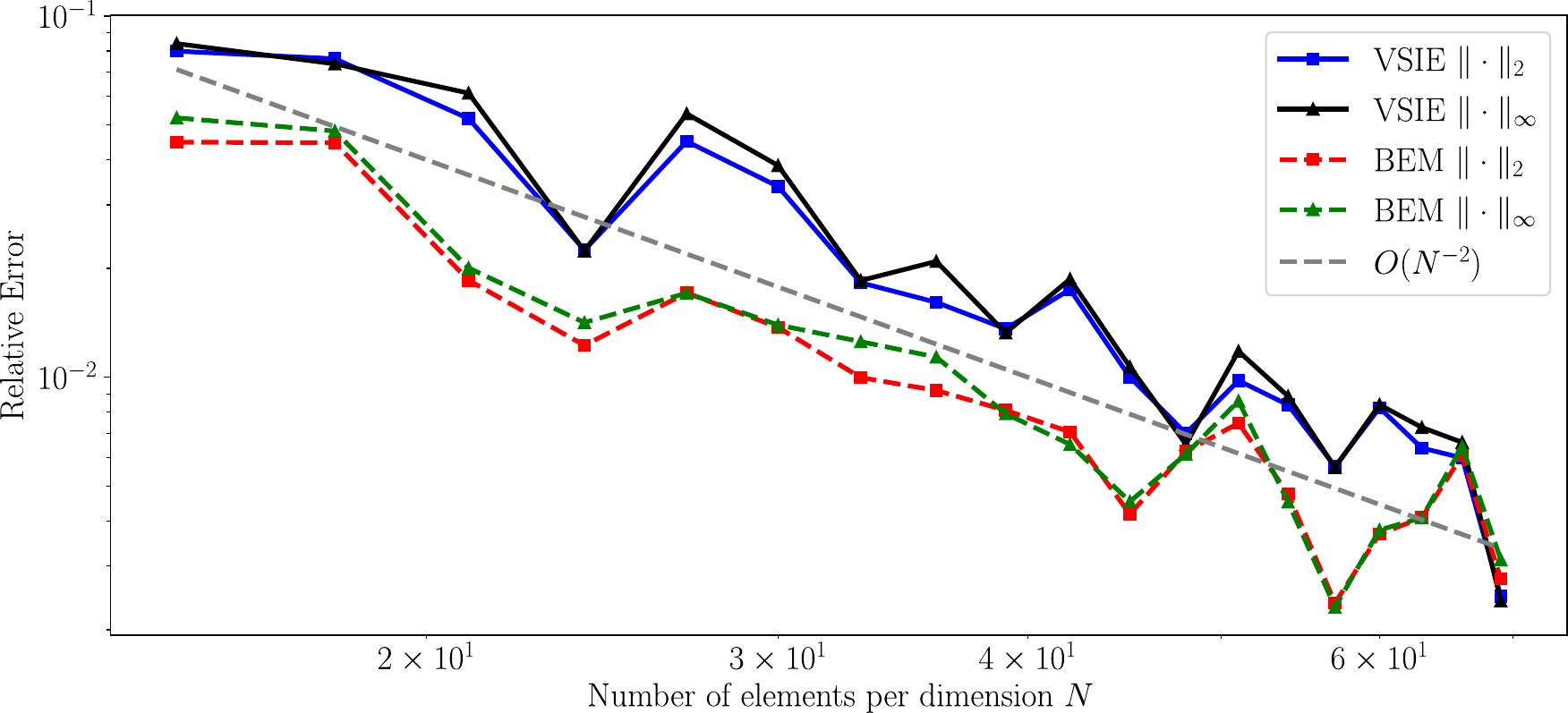}
    \caption{The mesh convergence study on a sphere shows the numerical error relative to the analytical solution versus the number of elements per dimension $N$. For this benchmark, we have $N = N_x = N_y = N_z$. The dashed line depicts second-order convergence $(\mathcal O(N^{-2}))$.}
    \label{fig_sphere_loglog}
\end{figure}

To test the consistency of our numerical approach, let us study mesh convergence. For this purpose, we calculated the relative error of the pressure field at different meshes and a fixed frequency. Figure~\ref{fig_sphere_loglog} confirms that the numerical errors reduce consistently when finer meshes are used. In general, the VIE tends to have slightly higher errors than the BEM, given the same number of grid elements per wavelength. We attribute this to the staircase effect of the cuboid mesh and the discretization strategy. That is, the implemented BEM algorithm uses triangular meshes, P1~elements, and higher-order quadrature, whereas our custom implementation of the VIE uses low-order approximations on cuboids. The errors of the BEM and VSIE decrease at a rate of $\mathcal O(N^{-2})$ with respect to $N$, the number of elements per dimension. This second-order convergence verifies the accuracy of the VSIE method. 

\subsection{Piecewise-constant material density}

Let us consider a geometry with two concentric cubes embedded in an unbounded region, as depicted in Fig.~\ref{fig_geometry}. We select piecewise-constant material parameters that include a density jump. Hence, the VSIE~\eqref{eq_vsie} involves the single-layer volume integral over the interior domains and the double-layer boundary integral over the material interface. However, the adjoint-double-layer volume integral vanishes due to the piecewise-constant density. As the density function~\eqref{eq_density}, we use $\rho_0 = \rho_1 = 1000$~kg/m$^3$ and $\rho_2 = 500$~kg/m$^3$. The speed of sound function~\eqref{eq_speed_of_sound} is taken as $c_0 = 1482.3$~m/s in the exterior to resemble water, and we use $c_1 = c_2 = c_0/1.2$, a slightly lower speed of sound in the interior domains. The frequency is such that the wavelength equals 3~mm in the exterior domain. The subdomain $\Omega_1$ is a cube with sides of 7~mm and $\Omega_2$ a cube with 5~mm sides, both centered in the global origin.

In the absence of an analytical solution, we compare the VSIE results with those from the BEM. Precisely, the acoustic field is calculated on the slice $z=0$ with 10 elements per wavelength on the rectangle $[-3.5,3.5]^2$~mm. Figure~\ref{fig_piecewise_constant_field} shows the field calculated by the VSIE and BEM and confirms the accuracy of the numerical methods. The overall wave patterns are the same, with only slight differences at the material interface, where the VSIE displays numerical artifacts. These numerical inaccuracies are due to the low-order numerical approximations performed in our current implementation of the VSIE. In any case, these localized inaccuracies do not significantly affect other regions. It's also clear that both models accurately represent the acoustic diffraction patterns.

\begin{figure}[htbp]
    \centering
    \includegraphics[width=0.7\textwidth]{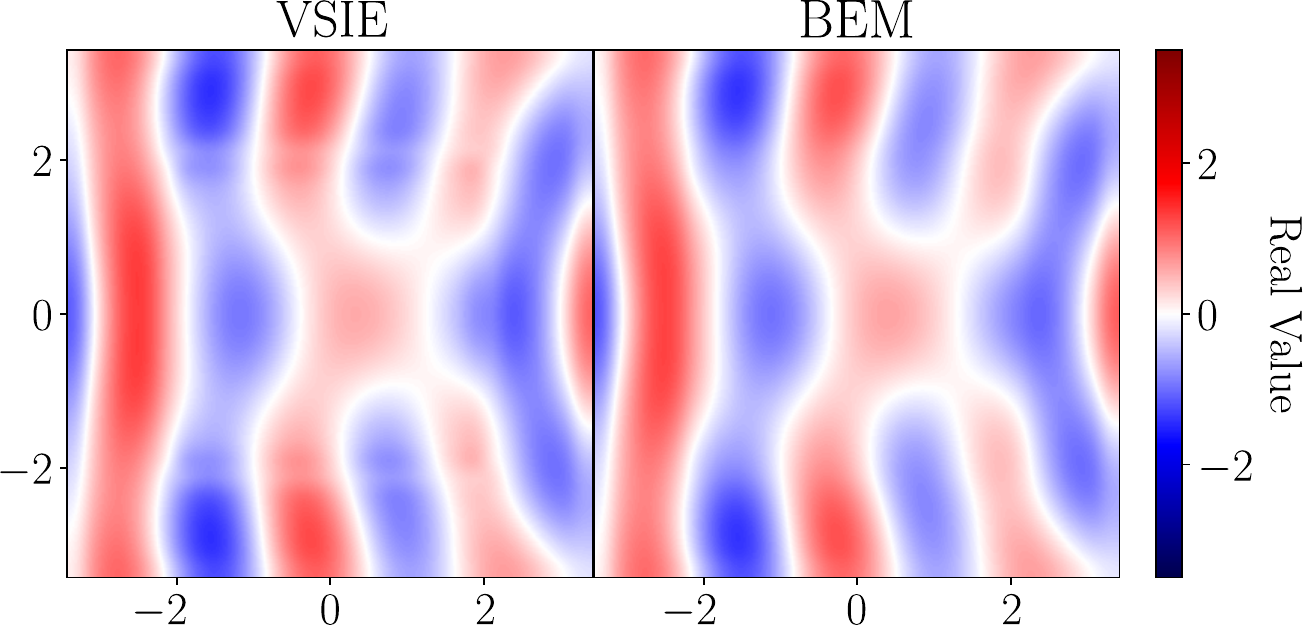}
    \caption{The real part of the acoustic field calculated by the VSIE (left panel) and the BEM (right panel), at $z=0$. This benchmark has a wavespeed jump at the outer interface and a density jump at the inner interface of the nested cubes. The meshes use 10 elements per interior wavelength.}
    \label{fig_piecewise_constant_field}
\end{figure}

Let us also study mesh convergence for this benchmark. The BEM uses a triangular surface mesh at the boundaries $\Gamma_1$ and $\Gamma_2$, which can be generated with arbitrary mesh resolution. In contrast, the VSIE uses a voxel mesh across the volumetric domain. When using uniform voxels, the geometric interface~$\Gamma_2$ may not coincide with faces between voxels. This misalignment can cause significant numerical errors. To avoid this geometric error, we use a nonuniform cuboid mesh that exactly matches the geometric interfaces at voxel faces; see Figure~\ref{fig_grid}. Hence, the voxel mesh has slightly different resolutions in the two concentric cubes. This still allows for consistent mesh refinement, and we report the number of elements per wavelength for the largest voxel.

\begin{figure}[htbp]
    \centering
    \includegraphics[width=0.8\textwidth]{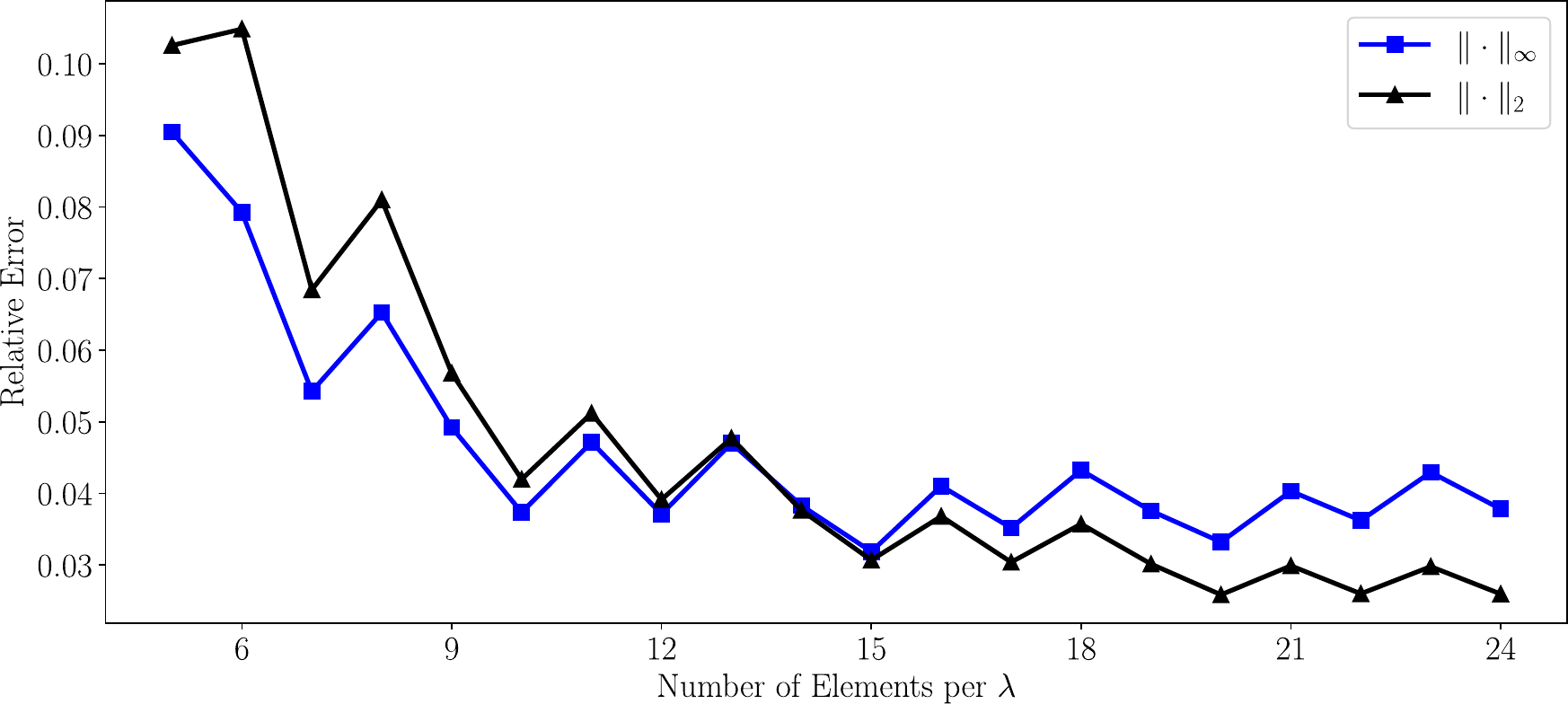}
    \caption{The convergence of the relative difference between VSIE and BEM in the $L^2$ and $L^\infty$ norms with respect to the number of elements per interior wavelength~$\lambda$. The benchmark has piecewise-constant material parameters.}
    \label{fig_piecewise_constant_convergence}
\end{figure}

In the absence of an analytical solution, we calculate the relative error of the VSIE using the BEM solution as the benchmark in Eq.~\eqref{eq_error}. Figure~\ref{fig_piecewise_constant_convergence} confirms that the fields calculated by the two numerical methods quickly match with only a few grid elements per wavelength. For example, the relative difference is below 8\% with 8 or more elements per wavelength. This is slightly higher than in the previous benchmark on the sphere, where the density was globally constant. Also, notice that the error is again concentrated near the material interface with the density discontinuity. In the end, the relative difference between the VSIE and BEM drops below 5\%, which is a reasonable threshold for comparing two numerical algorithms.

\subsection{Heterogeneous density}

For the next benchmark, we consider a single heterogeneous domain~$\Omega_1 = [-3.5, 3.5]^3$~mm, for which the VSIE simplifies to Eq.~\eqref{eq_vsie_single_domain}. We use $c_0 = 1482.3$~m/s as the exterior speed of sound and $c_1 = c_0/1.2$ in the interior. The density will be taken as heterogeneous in the interior. Specifically,
\begin{align*}
  \rho(x,y,z) =
  \begin{cases}
      1000 & \text{if } (x,y,z) \in \Omega_0, \\
      1000 - 800 \sin\left(\frac{\pi}{3.5} x\right) \sin\left(\frac{\pi}{3.5} y\right) \cos\left(\frac{\pi}{7} z\right) & \text{if } (x,y,z) \in \Omega_1.
      \end{cases}
\end{align*}
Notice that the density is a smooth function, with values between 200 and 1800~kg/m$^3$, and has no jumps across~$\Gamma_1$. We select a frequency of $f = c_0/4$~kHz, which corresponds to a wavelength of 4~mm in the exterior.

Since the density is heterogeneous, the BEM cannot be used to solve this scenario. Instead, we compare the VSIE with a FEM-BEM algorithm, where the BEM is defined on $\Gamma_1$ and the FEM on $\Omega_1$. Standard Johnson-Nédelec coupling was used, as explained in Section~\ref{sec_fembem}.

\begin{figure}[ht]
    \centering
    \includegraphics[width=0.7\textwidth]{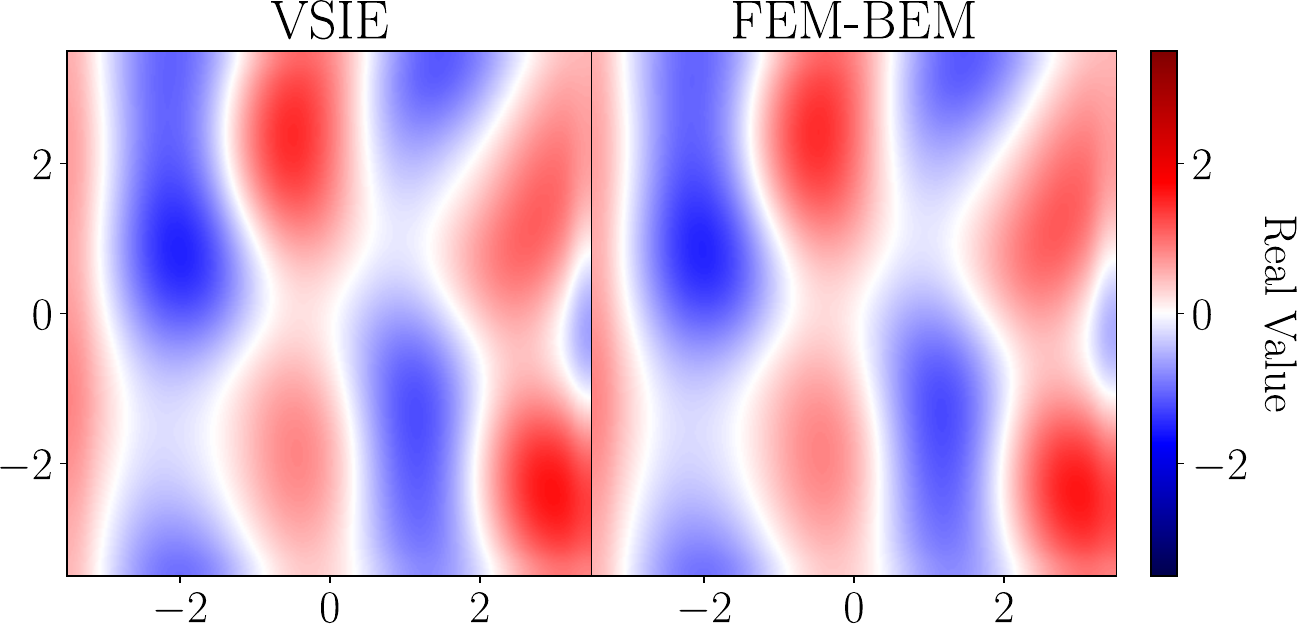}
    \caption{The real part of the solution for the benchmark with heterogeneous materials at $z=0$, calculated with the VSIE (left panel) and FEM-BEM (right panel). The meshes use 10 elements per interior wavelength for VSIE and 20 for FEM-BEM.}
    \label{fig:benchmark3}
\end{figure}

The results in Figure~\ref{fig:benchmark3} clearly show that the acoustic fields are consistent. The heterogeneous density causes nonsymmetric diffraction and propagation patterns. Qualitatively, no differences between the VSIE and FEM-BEM are visible.

As before, we also quantified the relative error in a mesh convergence study. However, a direct comparison between the VSIE and FEM-BEM with the same number of elements per wavelength proved misleading. The FEM, and also coupled FEM-BEM, typically need a finer mesh to achieve the same accuracy as pure integral equation approaches like BEM and VSIE. Based on our computational experience with these benchmarks, we decided to compare VSIE simulations at resolution~$h$ with FEM-BEM simulations at $h/2$. Hence, FEM-BEM uses a two-times-finer mesh than VSIE.

\begin{figure}[htbp]
    \centering
    \includegraphics[width=0.8\textwidth]{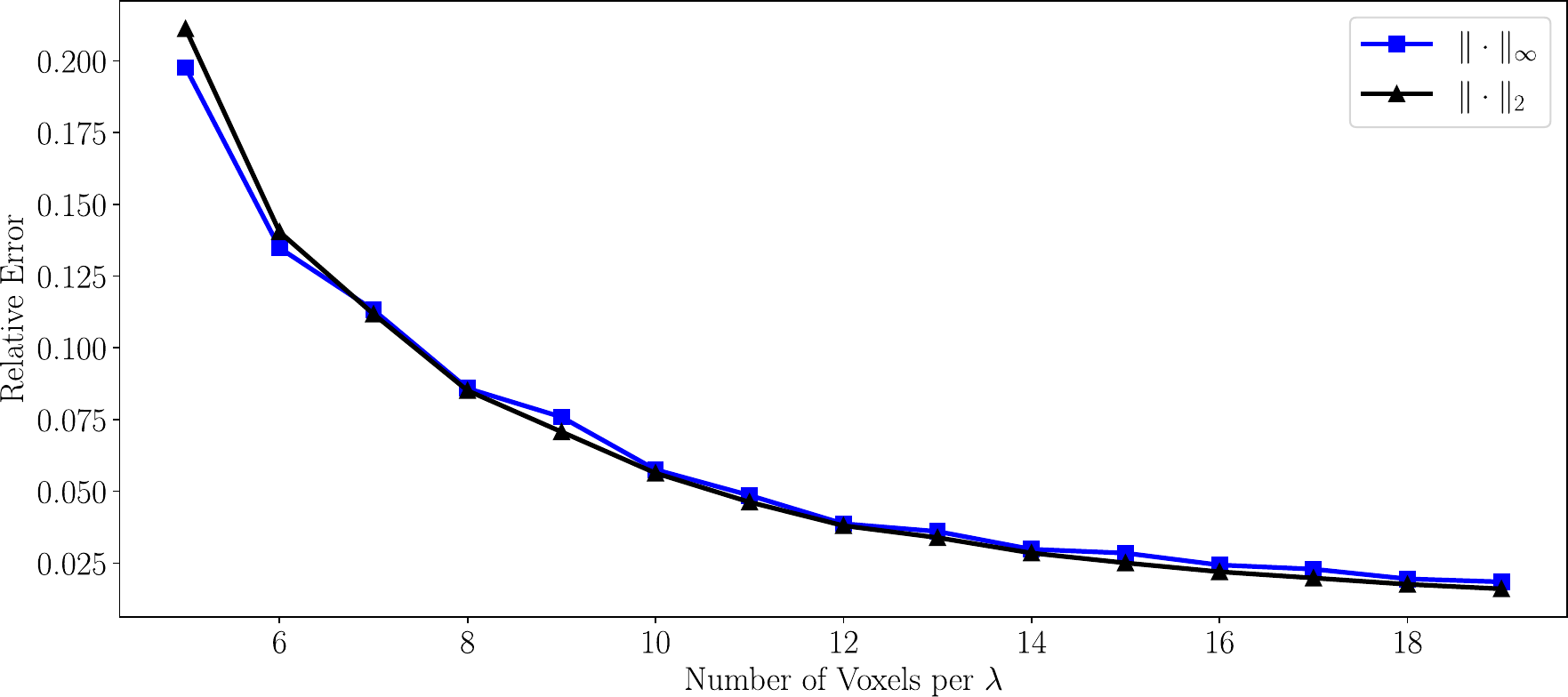}
    \caption{The convergence of the relative difference in the $L^2$ and $L^\infty$ norms with respect to mesh refinement for Benchmark 3. The horizontal axis denotes the mesh refinement level $N$. For the VSIE method, $N$ represents the number of voxels per interior wavelength. For the FEM-BEM method, we use $2N$ elements per wavelength.}
    \label{fig:benchmark3error}
\end{figure}

Figure~\ref{fig:benchmark3error} shows the mesh convergence of the relative error between VSIE and FEM-BEM on the heterogeneous benchmark. The relative difference between these different numerical methods consistently reduces with finer meshes. The error drops below 5\% at refinement level $N=11$. That is, 11 elements per wavelength in the VSIE and 22 for FEM-BEM. At the finest meshes, the algorithms are sufficiently accurate for most practical purposes.

\subsection{Heterogeneous density with discontinuity}

Our fourth benchmark will showcase the full capability of VSIE to model acoustic wave propagation through nested domains with heterogeneous materials and jumps across interfaces. We will consider concentric cubes of side lengths 7~mm and 5~mm for $\Omega_1$ and $\Omega_2$, respectively. We use $c_0 = 1482.3$~m/s as the exterior speed of sound and $c_1 = c_2 = c_0/1.2$ in the interior. As the density function, we use
\begin{align*}
  \rho(x,y,z) =
  \begin{cases}
      1000 & \text{if } (x,y,z) \in \Omega_0,  \\
      1000 & \text{if } (x,y,z) \in \Omega_1,  \\
      2000 - 800 \sin\left(\frac{\pi}{2.5} x\right) \sin\left(\frac{\pi}{2.5} y\right) \cos\left(\frac{\pi}{5} z\right) & \text{if } (x,y,z) \in \Omega_2.
      \end{cases}
\end{align*}
Notice that this means that the density is heterogeneous in the inner cube and has a jump between 1000 and 2000~kg/m$^3$ across $\Gamma_2$, the interface between the two nested cubes. The wavespeed has a jump across the outer boundary $\Gamma_1$. The frequency is $f = c_0 / 2.5$~kHz, resulting in an exterior wavelength of 2.5~mm.

The VSIE's solution will be compared with a coupled FEM-BEM system, where the FEM covers the heterogeneous inner cube~$\Omega_2$, and two BEM equations are formulated on $\Gamma_2$ and $\Gamma_1$. The resulting FEM-BEM-BEM coupled system will be solved at half the VSIE mesh resolution. As before, this finer mesh for FEM-BEM provides a fairer comparison with the VSIE.

The results in Figure~\ref{fig:benchmark4} again confirm the consistency of the VSIE. The acoustic fields are qualitatively the same for the benchmark involving both heterogeneity and high material contrasts. The most visible deficiency of the VSIE is near the high-contrast interface. As in previous benchmarks, the near-interface inaccuracies do not significantly affect the global solution. Furthermore, postprocessing of the VSIE's solution can improve the acoustic field. That is, highly accurate field evaluation can be performed at arbitrary locations by applying the representation formula~\eqref{eq_vsie} to the discrete solution.

\begin{figure}[htbp]
    \centering
    \includegraphics[width=0.7\textwidth]{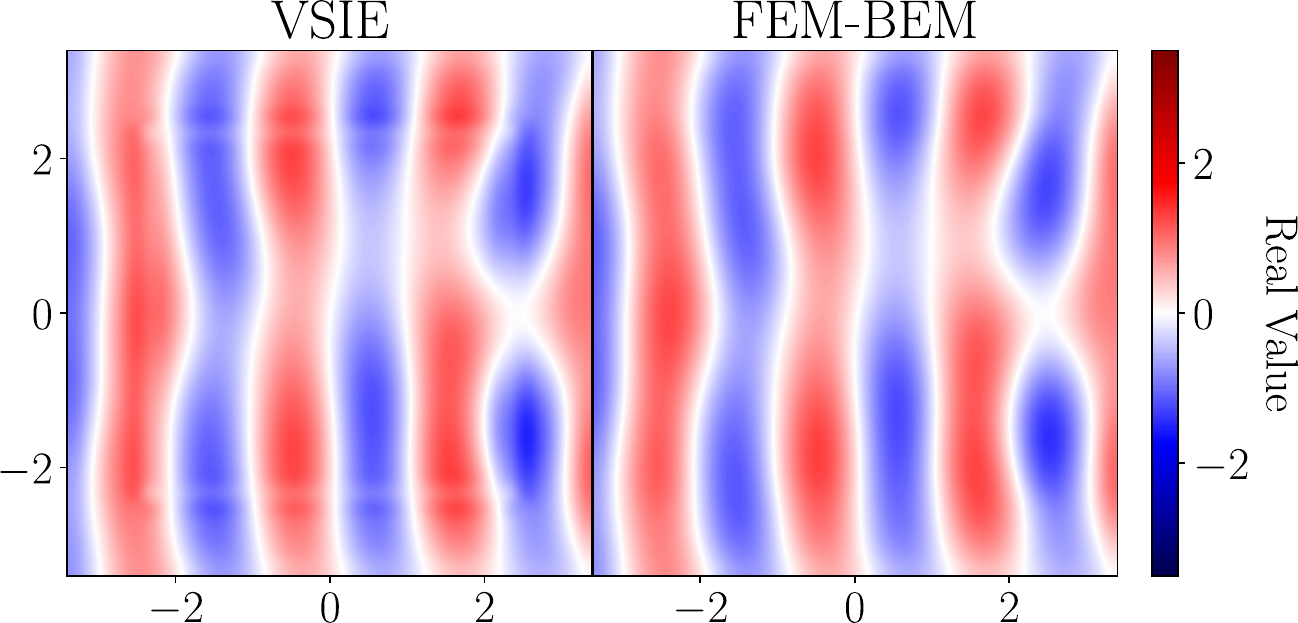}
    \caption{The real part of the solution for Benchmark~4 at $z=0$, calculated with the VSIE (left panel) and FEM-BEM-BEM (right panel). The meshes use 10 elements per interior wavelength for VSIE and outer BEM, and 20 for FEM-BEM.}
    \label{fig:benchmark4}
\end{figure}

For the mesh convergence study, we follow the same approach as before. The frequency used in the study is $f=c_0/3.5$ kHz. A nonuniform cuboid mesh is used for the VSIE to guarantee matching geometric and grid interfaces, and the FEM-BEM solution is calculated at twice the number of elements per wavelength. Figure~\ref{fig:benchmark4error} confirms mesh convergence. However, for this more challenging benchmark, we observe a slower convergence than previously. The differences start leveling off after 12 elements per wavelength in the VSIE. Still, the error already reached the 5\% threshold in the Frobenius norm. The maximum-norm error remains slightly higher, likely due to localized errors near the high-contrast interface.

\begin{figure}[htbp]
    \centering
    \includegraphics[width=0.8\textwidth]{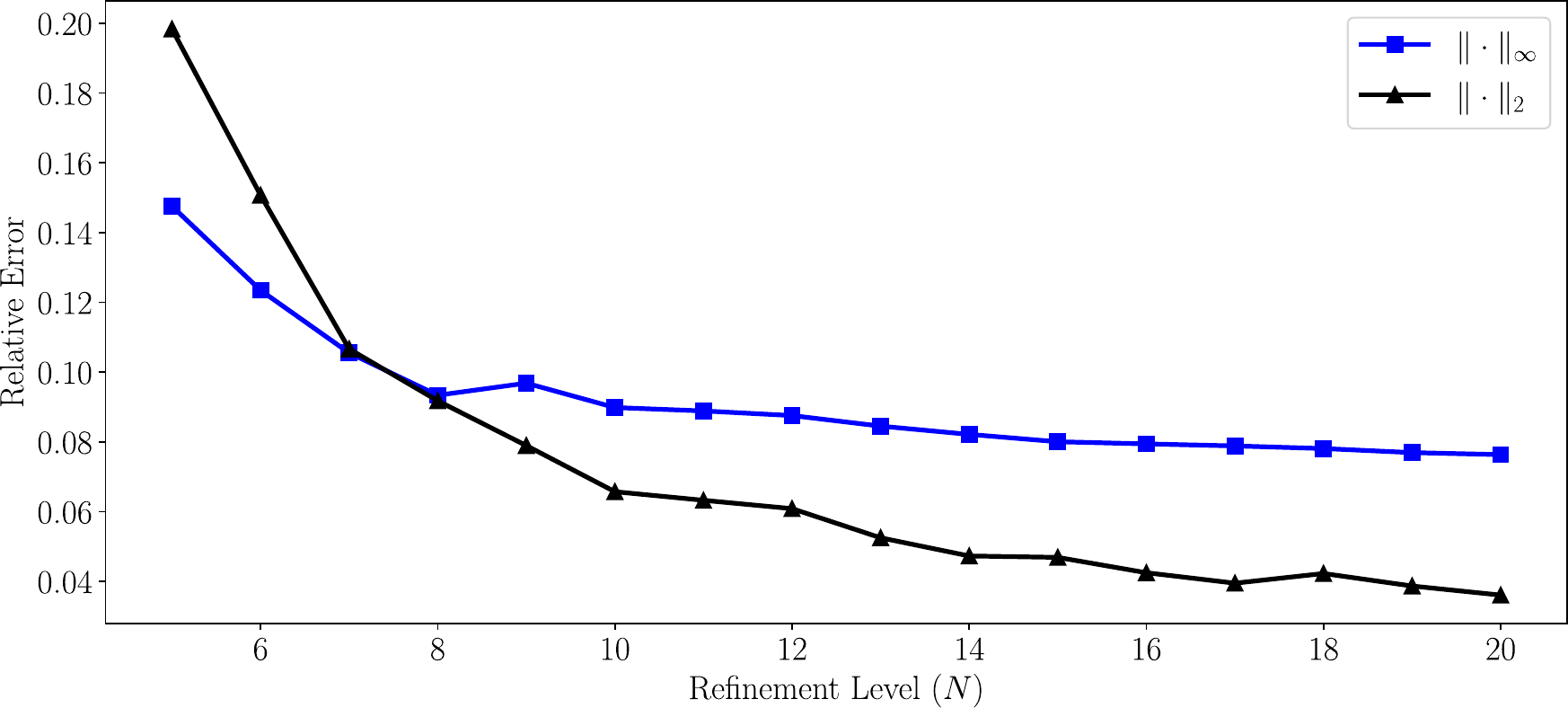}
    \caption{The convergence of the relative difference with norm $L^2$ and $L^\infty$ with respect to the number of points per wavelength for Benchmark 4. The horizontal axis denotes the mesh refinement level $N$. For the VSIE method, $N$ represents the number of voxels per interior wavelength. For the FEM-BEM method, the BEM mesh on $\Gamma_1$ uses $N$ elements per wavelength, while the BEM mesh on $\Gamma_2$ and the FEM mesh use $2N$ elements per wavelength.}
    \label{fig:benchmark4error}    
\end{figure}

To highlight the complexity of this benchmark, we present a 3D visualization of the solution in Figure~\ref{fig_benchmark4_3D}. It uses the same wavespeed and density functions, with the incident wave field~\eqref{eq_planewave} propagating in direction $\hat{\mathbf{d}} = \sqrt{1/2}\,(1,1,0)$ and the frequency was set to $f = c_0 / 1.4$~kHz to achieve an exterior wavelength of 1.4~mm. The VSIE was solved on a 10-voxel-per-wavelength grid, containing 226\,981~voxels in total.

\begin{figure}[htbp]
    \centering
    \includegraphics[height=60mm, trim=550 220 530 300, clip]{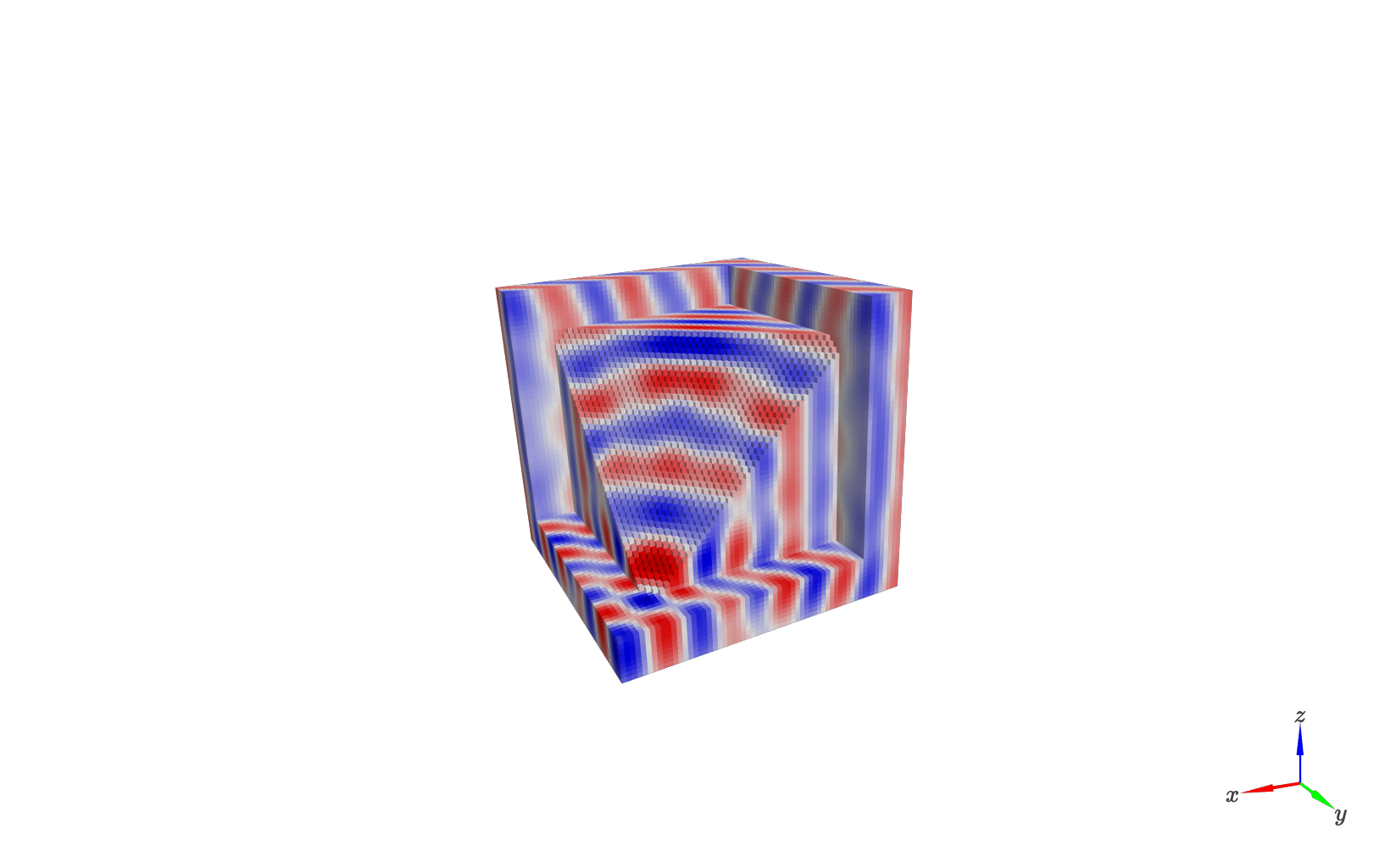}
    \caption{The real part of the solution for the pressure field in Benchmark~4. The 3D visualization uses a single color value per grid voxel and diagonal clipping of the geometry.}
    \label{fig_benchmark4_3D}     
\end{figure}

Figure~\ref{fig_benchmark4_3D} provides an exploded view of the pressure field inside the bounded domains. Three sides of the intermediate domain~$\Omega_1$ are presented, with the other three removed for visibility purposes. The diagonal clipping of the nested cube~$\Omega_2$ allows visualization of the pressure field inside the heterogeneous material, as well as the fine voxel mesh used for the benchmark. The field values are constant for each voxel, as in the discretization.

\section{Discussion}

We restricted this manuscript to a geometry with two nested domains embedded in an unbounded region. While this topology includes a wider class of geometries than single domains, it's still limited. However, the same methodology can easily be extended to more nested domains by following the same design principles. Furthermore, multiple domains are implicitly covered, since the bounded domains need not be connected. Future research will target VSIE formulations and software code that flexibly handle multiple and nested domains.

Verifying novel methodologies against distinct numerical methods is often an underappreciated yet essential activity for establishing credibility in computational engineering. This study presented benchmarking exercises comparing the VSIE with analytical solutions, BEM simulations for piecewise-constant materials, and FEM-BEM simulations for locally heterogeneous materials in unbounded domains. The relative differences consistently diminished with mesh refinement, thus confirming the accuracy of the numerical techniques. However, convergence stalled for the more complex benchmarks. We expect the differences to drop further with more accurate VSIE implementations that include Galerkin discretization with piecewise linear elements, tetrahedral meshes, and higher-order numerical quadrature. Also, when increasing the frequency further, our ambiguous factor of two in mesh resolution between VSIE and FEM-BEM needs to be reconsidered, as both methods behave differently with respect to pollution effects in the high-frequency regime.

Finally, we purposefully did not report computation time or memory usage for the benchmarks. A fair comparison of computational performance between different numerical methods requires a thorough benchmark design that accounts for algorithmic details such as element type, quadrature schemes, numerical linear algebra, the number of elements per wavelength, and acceleration methods like fast multipole and matrix compression. Furthermore, comparisons must include implementation details such as programming language, parallelization scheme, and hardware characteristics.

\section{Conclusions}

This study successfully demonstrated two innovations regarding the VSIE algorithm. First, VSIE formulations for nested domains with heterogeneous materials and jumps across interfaces were derived in detail. Second, the accuracy of the VSIE was benchmarked on a suite of test cases against the BEM and FEM-BEM coupled methods. The computational benchmarks verified numerical accuracy, as mesh-convergence studies show consistent acoustic fields across challenging scenarios.

\section*{Acknowledgments}

\subsection*{Software and data availability}
The Python code to reproduce the numerical results is available on GitHub (\href{https://github.com/daniloaballayf/BenchmarkingVSIEAcoustics}{github.com/daniloaballayf/BenchmarkingVSIEAcoustics}).

\subsection*{Declaration on generative AI}
During the preparation of this work, the authors used Google Gemini to improve the manuscript's language and readability. The authors reviewed and edited the content as needed and take full responsibility for the article's content.

\subsection*{Funding sources}
This work was financially supported by the Agencia Nacional de Investigación y Desarrollo (ANID), Chile [FONDECYT 1230642].

\bibliographystyle{unsrtnat}
\bibliography{refs.bib}

\end{document}